\documentclass[final]{amsart} 
\usepackage{amsmath}
\usepackage{amssymb}
\usepackage{amsthm}
\usepackage{amscd}
\usepackage{mathtools}
\usepackage{todonotes}
\usepackage{bbm}
\usepackage{bm}
\usepackage{comment}
\usepackage{mathrsfs}
\usepackage[notcite,notref]{showkeys}
\usepackage{cite}
\usepackage[colorlinks=true,linkcolor=black,citecolor=black,urlcolor=blue]{hyperref}
\usepackage{tikz}
\usepackage{tikz-cd}
 
\newtheorem{theorem}{Theorem}[section]
\newtheorem*{theorem*}{Theorem}
\newtheorem{conjecture}[theorem]{Conjecture}
\newtheorem*{conjecture*}{Conjecture}

\newtheorem*{question*}{Question}

\newtheorem{lemma}[theorem]{Lemma}
\newtheorem*{lemma*}{Lemma}

\newtheorem{proposition}[theorem]{Proposition}
\newtheorem*{proposition*}{Proposition}

\newtheorem{corollary}[theorem]{Corollary}
\newtheorem*{corollary*}{Corollary}

\theoremstyle{definition}

\newtheorem{definition}[theorem]{Definition}
\newtheorem*{definition*}{Definition}
\newtheorem{hypotheses}[theorem]{Hypotheses}
\newtheorem*{hypotheses*}{Hypotheses}
\newtheorem{hypothesis}[theorem]{Hypothesis}
\newtheorem*{hypothesis*}{Hypothesis}

\newtheorem*{assumption*}{Assumption}
\newtheorem{remark}[theorem]{Remark}
\newtheorem{remarks}[theorem]{Remarks}

\newtheorem*{example*}{Example}
\newtheorem*{examples*}{Examples}

\newcommand{\twomat}[4]{\begin{pmatrix} #1 & #2 \\ #3 & #4 \end{pmatrix}}

\renewcommand{\bar}{\overline}
\usepackage{mathtools}

\renewcommand{\AA}{\mathbb{A}}

\newcommand{\CC}{\mathbb{C}}

\newcommand{\FF}{\mathbb{F}}
\newcommand{\GG}{\mathbb{G}}

\newcommand{\NN}{\mathbb{N}}

\newcommand{\QQ}{\mathbb{Q}}
\newcommand{\RR}{\mathbb{R}}

\newcommand{\TT}{\mathbb{T}}

\newcommand{\ZZ}{\mathbb{Z}}

\newcommand{\Ebar}{\overline{E}}
\newcommand{\FFbar}{\overline{\FF}}
\newcommand{\Icbar}{\overline{\Ic}}
\newcommand{\Mbar}{\overline{M}}
\newcommand{\Msbar}{\overline{\Ms}}
\newcommand{\Nbar}{\overline{N}}
\newcommand{\Nsbar}{\overline{\Ns}}

\newcommand{\Rbar}{\overline{R}}

\newcommand{\Sbar}{\overline{S}}

\newcommand{\Lambdabar}{\overline{\Lambda}}
\newcommand{\psibar}{\overline{\psi}}
\newcommand{\sigmabar}{\overline{\sigma}}

\newcommand{\Bc}{\mathcal{B}}
\newcommand{\Cc}{\mathcal{C}}
\newcommand{\Dc}{\mathcal{D}}

\newcommand{\Fc}{\mathcal{F}}
\newcommand{\Gc}{\mathcal{G}}
\newcommand{\Hc}{\mathcal{H}}
\newcommand{\Ic}{\mathcal{I}}

\newcommand{\Lc}{\mathcal{L}}

\newcommand{\Oc}{\mathcal{O}}
\newcommand{\Pc}{\mathcal{P}}
\newcommand{\Qc}{\mathcal{Q}}
\newcommand{\Rc}{\mathcal{R}}

\newcommand{\As}{\mathscr{A}}
\newcommand{\Bs}{\mathscr{B}}
\newcommand{\Cs}{\mathscr{C}}

\newcommand{\Ms}{\mathscr{M}}
\newcommand{\Ns}{\mathscr{N}}

\newcommand{\Rs}{\mathscr{R}}

\newcommand{\Vs}{\mathscr{V}}

\newcommand{\Xf}{\mathfrak{X}}

\newcommand{\af}{\mathfrak{a}}

\newcommand{\mf}{\mathfrak{m}}
\newcommand{\nf}{\mathfrak{n}}

\newcommand{\pf}{\mathfrak{p}}
\newcommand{\qf}{\mathfrak{q}}

\newcommand{\mftilde}{\widetilde{\mf}}
\newcommand{\TTtilde}{\widetilde{\TT}}

\newcommand{\rarrow}{\rightarrow}

\newcommand{\onto}{\twoheadrightarrow}
\newcommand{\into}{\hookrightarrow}
\newcommand{\isomto}{\xrightarrow{\sim}}
\newcommand{\longisomto}{\xrightarrow{\;\sim\;}}

\newcommand{\rhobar}{\bar{\rho}}
\newcommand{\invlim}{\varprojlim}
\newcommand{\dirlim}{\varinjlim}

\newcommand{\End}{\operatorname{End}}
\newcommand{\Hom}{\operatorname{Hom}}

\newcommand{\Tor}{\operatorname{Tor}}
\newcommand{\Frob}{\operatorname{Frob}}

\newcommand{\Gal}{\operatorname{Gal}}

\newcommand{\Ind}{\operatorname{Ind}}

\newcommand{\Spec}{\operatorname{Spec}}

\newcommand{\proj}{\operatorname{proj}}
\newcommand{\tr}{\operatorname{tr}}

\newcommand{\Nm}{\operatorname{Nm}}
\newcommand{\im}{\operatorname{im}}
\newcommand{\rank}{\operatorname{rank}}

\newcommand{\Sym}{\operatorname{Sym}}
\newcommand{\Ann}{\operatorname{Ann}}
\newcommand{\ch}{\operatorname{char}}

\newcommand{\St}{\operatorname{St}}

\newcommand{\ann}{\operatorname{ann}}

\newcommand{\Art}{\operatorname{Art}}

\newcommand{\univ}{\operatorname{univ}}

\newcommand{\fin}{\operatorname{fin}}

\newcommand{\ab}{\operatorname{ab}}

\newcommand{\triv}{\mathbbm{1}}

\title[Ihara's Lemma for Shimura curves]{Ihara's Lemma for Shimura curves over totally real fields via patching.}
\author{Jeffrey Manning}
\address{Math Sciences Building 6164,
UCLA Mathematics Department,
Los Angeles, CA 90095,
USA}
\email{jmanning@math.ucla.edu}
\author{Jack Shotton}
\address{Department of Mathematical Sciences,
Durham University,
Lower Mountjoy,
Stockton Road,
Durham DH1 3LE,
UK}
\email{jack.g.shotton@durham.ac.uk}
\usepackage[utf8]{inputenc}
\usepackage{mathscinet}

\newcommand{\et}{\mathrm{\acute{e}t}}

\newcommand{\nr}{\mathrm{nr}}
\newcommand{\unip}{\mathrm{unip}}
\newcommand{\ps}{\mathrm{ps}}
\newcommand{\loc}{\mathrm{loc}}

\newcommand{\patch}{\mathscr{P}}
\newcommand{\uf}{\mathfrak{F}}
\newcommand{\prF}{\mathfrak{Z}}

\newcommand{\uprod}[1]{\mathcal{U}\!\left(#1\right)}

\newcommand{\wP}{\mathfrak{wP}}
\newcommand{\Ab}{\mathbf{Ab}}
\newcommand{\finAb}{\mathbf{finAb}}

\newcommand{\ds}{\displaystyle}

\begin{document}
\begin{abstract}
  We prove Ihara's lemma for the mod $l$ cohomology of Shimura curves, localized at a maximal ideal of the Hecke
  algebra, under a large image hypothesis on the associated Galois representation.  This was proved by Diamond and
  Taylor, for Shimura curves over $\QQ$, under various assumptions on $l$. Our method is totally different and can avoid
  these assumptions, at the cost of imposing the large image hypothesis.  It uses the Taylor--Wiles method, as
  improved by Diamond and Kisin, and the geometry of integral models of Shimura curves at an auxiliary prime.
\end{abstract}
\maketitle

\section{Introduction}

Let $\Gamma = \Gamma_0(N)$ be the usual congruence subgroup of $SL_2(\ZZ)$, for some $N \geq 1$, and let $p$ be a prime
not dividing $N$.  Write $\Gamma' = \Gamma \cap \Gamma_0(p)$.  If $X_{\Gamma}$ and $X_{\Gamma'}$ are the compactified
modular curves of levels $\Gamma$ and $\Gamma'$, then there are two degeneracy maps
\[ \pi_1, \pi_2 : X_{\Gamma'} \rarrow X_\Gamma\]
induced by the inclusions $\Gamma' \into \Gamma$ and $\twomat{p}{0}{0}{1}\Gamma' \twomat{p}{0}{0}{1}^{-1} \into
\Gamma$.  If $l$ is another prime, then we have a map 
\[\pi^* = \pi_1^* + \pi_2^* : H^1(X_\Gamma, \FF_l)^2 \rarrow H^1(X_{\Gamma'}, \FF_l).\]
As a consequence of a result of Ihara --- \cite{MR0399105} Lemma~3.2, and see also the proof of \cite{MR804706}
Theorem~4.1 --- the kernel of $\pi^*$ may be determined.  In particular:
\begin{theorem*}[Ihara's Lemma] \label{thm:classical} If $\mf$ is a \emph{non-Eisenstein} maximal ideal of the Hecke
  algebra acting on these cohomology groups (that is, $\mf$ corresponds to an \emph{irreducible} Galois representation),
  then the map $\pi^*$ is injective after localizing at $\mf$.\footnote{In fact, if we instead take
    $\Gamma = \Gamma_1(N)$ then $\pi^*$ is already injective.  For us, however, localizing at a maximal ideal of the
    Hecke algebra will be crucial.}
\end{theorem*}
This was used by Ribet in~\cite{MR804706} to prove a level-raising result for modular forms: if $f \in S_2(\Gamma)$ is a
cuspidal eigenform such that $\rhobar_f$ is irreducible and the Fourier coefficient $a_p$ satisfies
\[a_p \equiv \pm(1 + p) \pmod l,\]
then there is a cuspidal eigenform $g \in S_2(\Gamma')^{\text{$p$-new}}$ such that $\rhobar_f \cong \rhobar_g$.

Now suppose that $F$ is a totally real number field and that $D$ is a quaternion division algebra over $F$ ramified at
all but one infinite place. For $K \subset (D \otimes \AA_{F,f})^\times$ a compact open subgroup, $\pf$ a finite place
of $F$ at which $K$ and $D$ are unramified, and $l$ a prime, there is an obvious (conjectural) generalisation of
Theorem~\ref{thm:classical} with $X_\Gamma$ replaced by the Shimura curve $X_K$.  We refer to this as ``Ihara's Lemma at
$\pf$ for $X_K$, localized at $\mf$''; it depends on $K$ and on a maximal ideal $\mf$ of the Hecke algebra acting on
$H^1(X_K, \FF_l)$, to which is associated a Galois representation $\rhobar_\mf : G_F \rarrow GL_2(\bar{\FF}_l)$.  The
purpose of this paper is to prove:
\begin{theorem} Suppose that $l > 2$ and that the image of $\rhobar_{\mf}$ contains a subgroup of $GL_2(\bar{\FF}_l)$
  conjugate to $SL_2(\FF_l)$ (and satisfies an additional Taylor--Wiles hypothesis if $l = 5$ and $\sqrt{5} \in F$).

  Then Ihara's Lemma at $\pf$ for $X_K$, localized at $\mf$, is true.
\end{theorem}

Ihara's method of proof does not generalise, since it relies on the ``congruence subgroup property of
$SL_2(\ZZ[\frac{1}{p}])$'', the analogue of which is a longstanding conjecture of Serre in the quaternionic case.
In~\cite{MR1262939}, Diamond and Taylor overcame this difficulty for Shimura curves over $\QQ$ using the good reduction
of Shimura curves at $l$ and comparison of mod $l$ de Rham and \'{e}tale cohomology.  This necessitates various
conditions on $l$:
\begin{itemize}
\item $\pf$ does not divide $l$;
\item $D$ and $K$ must be unramified at $l$;
\item if the result is formulated with coefficients $\Sym^{k-2}\FF_l$, then the weight $k$ satisfies
  \[k \leq l - 1.\footnote{See the end of \cite{DiamondTaylor1994-Lifting} for $k = l-1$.}\]
\end{itemize}
It seems likely that the approach of~\cite{MR1262939} can be adapted to the totally real case with similar conditions on
$l$, as in Cheng's draft \cite{Cheng} (which the author tells us is not complete), but this has not yet been carried out in full detail.

Our method of proof is entirely different, and requires no such conditions on $l$.  On the other hand, we have to impose
a more stringent condition on $\rhobar_{\mf}$ --- rather than merely being irreducible, its image must contain the
subgroup $SL_2(\FF_l)$.

Our starting point is that Ihara's Lemma is known (and easy) for the ``Shimura sets'' associated to definite quaternion
algebras.  Following a strategy introduced by Ribet in~\cite{MR1047143} we introduce an auxiliary prime $\qf$, at which
both $K$ and $D$ are unramified.  Then there is a totally definite quaternion algebra $\bar{D}$ ramified at the same
finite places as $D$, together with $\qf$, and a compact open subgroup
$K^{\qf} \subset (\bar{D} \otimes \AA_{F,f})^\times$ agreeing with $K$ at all places besides $\qf$ and maximal at
$\qf$. Our goal will then be to reduce the statement of Ihara's Lemma for $X_K$ at $\mf$ to the corresponding (known)
statement for the Shimura set $Y_{K^{\qf}}$ corresponding to $K^{\qf}$.

The link between $X_K$ and $Y_{K^{\qf}}$ is given by the geometry of integral models of the Shimura curve
$X_{K_0(\qf)}$, with $\Gamma_0(\qf)$-level structure. Specifically, the special fibre of $X_{K_0(\qf)}$ at $\qf$
consists of two components, both of which are isomorphic to the special fibre of $X_K$, and has singularities at a
finite set of points which are in bijection with $Y_{K^\qf}$. This results in a filtration of
$H^1(X_{K_0(\qf)}, \FF_l)$ whose graded pieces are two copies of $H^0(Y_{K^\qf}, \FF_l)$ and one copy of
$H^1(X_K, \FF_l)^{\oplus 2}$. This idea has been extensively studied by Mazur, Ribet~\cite{MR1047143},
Jarvis~\cite{MR1669444} and others.

Unfortunately, the existence of this filtration does not directly imply any relation between the Hecke module structures
of $H^1(X_K, \FF_l)$ and $H^0(Y_{K^\qf}, \FF_l)$. For example, the filtration could be split (in the sense that
\[H^1(X_{K_0(\qf)},\FF_l)\cong H^1(X_{K},\FF_l)^{\oplus 2}\oplus H^0(Y_{K^\qf}, \FF_l)^{\oplus 2}\] as Hecke modules)
which would not impose any relations between $H^1(X_{K},\FF_l)$ and $H^0(Y_{K^\qf}, \FF_l)$. So in order to deduce
anything about $H^1(X_{K},\FF_l)$ and $H^0(Y_{K^\qf}, \FF_l)$, we need to have additional information about the Hecke
module structure of $H^1(X_{K_0(\qf)},\FF_l)$ and its interaction with the filtration.

The novelty of this paper, then, is to obtain this extra information.  It takes the form of a certain ``flatness''
statement, which we formulate and prove by using the Taylor--Wiles--Kisin patching method.  To our knowledge, this is
the first time that patching has been combined with the geometry of integral models in this way.

Briefly, the Taylor--Wiles--Kisin method considers a ring $R_\infty$, which is a power series ring over the completed
tensor product of various local Galois deformation rings, and relates the Hecke modules $H^1(X_{K},\FF_l)$,
$H^0(Y_{K^\qf}, \FF_l)$ and $H^1(X_{K_0(\qf)},\FF_l)$ to certain maximal Cohen--Macaulay ``patched'' modules over
$R_\infty$. Our method proves that the ``patched'' module corresponding to $H^1(X_{K_0(\qf)},\FF_l)$ is
flat\footnote{This is a slight simplification.} over some specific local deformation ring at the prime $\qf$. Using this
and some commutative algebra we are able to deduce Ihara's Lemma for $X_K$ from the corresponding result for
$Y_{K^{\qf}}$.

Our strategy for proving this flatness is inspired by Taylor's ``Ihara avoidance'' argument, used in the proof of the
Sato--Tate conjecture~\cite{Taylor2008-AutomorphyII}. We impose the condition that our auxiliary prime $\qf$ satisfies
$\Nm(\qf)\equiv 1\pmod{l}$, and consider a certain tamely ramified principal series deformation ring,
$R^{\ps}_{\qf} = R^{\ps}_{\rhobar_{\mf}|_{G_{F_{\qf}}}, \Oc}$, which is a quotient of the universal local deformation
ring $R_{\qf} = R^{\square}_{\rhobar_{\mf}|_{G_{F_{\qf}}},\Oc}$. The standard map\footnote{Suppressing minor issues due
  to framing and fixed determinants.} from $R_{\qf}$ to the mod $l$ Hecke algebra acting on $H^1(X_{K_0(\qf)},\FF_l)$
then factors through the quotient $R^{\ps}_{\qf}$, even though the map from $R_{\qf}$ to the integral Hecke algebra
acting on $H^1(X_{K_0(\qf)},\ZZ_l)$ does not.

In our situation, the assumption on the image of $\rhobar_\mf$ allows us to choose the auxiliary prime $\qf$ so that
\[\rhobar_\mf(\Frob_{\qf}) = \twomat{1}{1}{0}{1}.\]
In this case, the ring $R^{\ps}_{\qf}$ is a regular local ring\footnote{Provided that one carefully controls the
  ramification in the coefficient ring $\Oc$.} (a calculation carried out in \cite{MR3554238}). This is what enables us
to gain a foothold --- it is a standard principle going back to Diamond~\cite{MR1440309} that regular local deformation
rings give rise to important structural results about Hecke modules. We apply a version of the miracle flatness
criterion to prove that a particular patched module is flat over $R^{\ps}_{\qf}$, which is the key fact needed to make
our argument work.

The advantage of this argument, as opposed to that of \cite{MR1262939}, is that we do not need to make
any assumptions about the structure of the local deformation rings at primes dividing $l$, or indeed at any primes
besides $\qf$, beyond knowing that they have the correct dimension (a fact which certainly holds in the generality we
need). This is the reason we do not need to impose any of the restrictions on the prime $l$ appearing in earlier
results.

\subsection{Applications}
\label{sec:applications}

We briefly survey some of the applications of Ihara's Lemma (for modular or Shimura curves, or Shimura sets) that are in
the literature.

\subsubsection{Representation theoretic reformulation}
\label{sec:repthy}

Suppose that $K^{\pf} \subset (D \otimes \AA_{F,f}^{\pf})^\times$ is a compact open subgroup, and let 
\[ V = \varinjlim_{K_\pf}H^1(X_{K_{\pf}K^{\pf}}, \bar{\FF}_l)\] where the limit runs over compact open subgroups
$K_{\pf} \subset GL_2(F_{\pf})$.  Then $V$ is a smooth admissible representation of $GL_2(F_{\pf})$.  Suppose that
$\mf$ is a maximal ideal of the Hecke algebra acting on $V$.  Then we have:
\begin{proposition}
  Suppose that, for $K = K_{\pf}K^{\pf}$ with $SL_2(\Oc_{F,\pf}) \subset K_{\pf} \subset GL_2(\Oc_{F,\pf})$ a compact
  open subgroup, Ihara's Lemma is true for $X_K$ at $\mf$.  Then the representation $V_\mf$ of
  $GL_2(F_\pf)$ has no one-dimensional subrepresentations. \qed
\end{proposition}
\begin{remark}
  If $l \neq p$ then there is a notion of genericity for smooth representations of $GL_n(F_{\pf})$ (see, for instance,
  \cite{MR3250061}); when $n = 2$, the non-generic smooth irreducible representations are precisely the one-dimensional
  ones.  It is this ``no non-generic subrepresentations'' property that conjecturally generalises to higher rank (see
  \cite{ClozelHarrisTaylor2008-Automorphy}).
\end{remark}

\subsubsection{Freeness results}
\label{sec:freeness}

If $\TT$ is the algebra of Hecke operators acting on $\Hc = H_1(X_K, \Oc)_\mf$, \emph{including} those at primes at which
$K$ ramifies, then we can ask whether $\Hc$ is free as a $\TT$-module.  For modular curves results along these lines were
proved by Mazur, Ribet and others (see, for instance, \cite{MR1176206} Theorem~9.2 and \cite{calegari-geraghty}
Theorem~4.8).  In the case of Shimura curves, there are results starting with~\cite{MR1159117}.  Note that it is not
\emph{always} the case that $\Hc$ is free; in many cases this can be explained by the geometry of local deformation rings,
as in work of the first author \cite{Manning}.

In \cite{MR1440309} section~3.2, it is explained how the Taylor--Wiles method and a `numerical criterion' may be used to
prove freeness results at minimal and non-minimal level for modular curves (some limited freeness results for Shimura
curves are also given in \cite{MR1440309} section~3.3).  At non-minimal level, this relies crucially on Ihara's Lemma, and so using our
result we can extend these freeness results.   For instance, we have the following, in which $\Sigma$ denotes the set of
places where we are allowing non-minimal level.

\begin{theorem}\label{thm:mult 1}
  Let $F$ be a totally real number field, $D$ be a quaternion algebra over $F$ ramified at exactly one infinite
  place, $\Sigma$ a finite set of finite places of $F$, and $l > 2$ be a prime. 

  Let $K = \prod_v K_v \subset (D \otimes \AA_{F,f})^\times$ be a compact open subgroup and let $k \geq 2$ be an
  integer.  Let $\Hc = H_1(X_{K}, \Sym^{k-2}(\Oc_F^2 \otimes \ZZ_l))$, and let $\TT_K$ be the Hecke algebra acting on
  $\Hc$ generated by the $T_v$ and $S_v$ for $v \nmid l$ at which $K_v$ is maximal compact and $D$ is split, and the
  $U_v$ for each $v\in \Sigma$.

  Let $\mf$ be a maximal ideal of $\TT_{K}$ containing $l$.  Suppose that the Galois representation $\rhobar$ attached
  to $\mf$ has non-exceptional image, and that the following conditions hold.
  \begin{enumerate}
  \item For all finite places $v \mid l$ of $F$, $F_v/\QQ_l$ is unramified and $D$ is split at $v$.
  \item For all finite places $v \in \Sigma$ not dividing $l$, $D$ is split and $\rhobar$ is unramified at $v$.
  \item For all finite places $v \nmid l$ of $F$, $\rhobar|_{G_{F_v}}$ has minimal Artin conductor $n_v$ among all its twists by
    characters of $G_{F_v}$.
  \item For all finite places $v \nmid l$ of $F$ at which $D$ splits, either:
    \begin{itemize}
    \item $v \not \in \Sigma$ and $K_v = U_1(v^{n_v})$; or
    \item $v \in \Sigma$ and $K_v = U_1(v^2)$.
    \end{itemize}
(See (\ref{sec:U_0}) below for the definition of $U_1(v^n)$).
\item For all finite places $v$ of $F$ at which $D$ ramifies, $K_v$ is the group of units in a maximal order of
  $D \otimes F_v$, and if $\rhobar$ is unramified at $v$ then either:
    \begin{itemize}
    \item $\Nm(v)\not\equiv \pm 1 \pmod{l}$;
    \item $\Nm(v) \equiv 1 \pmod l$ and $\rhobar(\Frob_v)$ is not scalar; or
    \item $\Nm(v) \equiv -1 \pmod l$, and $\tr(\rhobar(\Frob_v)) \neq 0$.
    \end{itemize}
  \item If $v \nmid l$ is a place of $F$ at which $D$ splits and $\Nm(v) \equiv -1 \pmod l$, then either $\rhobar|_{G_{F_v}}$ is
    reducible or $\rhobar(I_{F_v})$ has order divisible by $l$.
  \item One of the following holds.
    \begin{itemize}
    \item (the \emph{Fontaine--Laffaille case}) $2 \leq k \leq l-1$ and $K_v$ is a maximal compact subgroup for each
      $v \mid l$; or
    \item (the \emph{ordinary case}) $k = 2$ and, for each $v \mid l$, either:
      $v \not \in \Sigma$ and $K_v$ is maximal compact; or $v \in \Sigma$, $F_v \cong \QQ_l$, $K_v = U_0(v)$, and
      $\rhobar|_{I_{F_v}} \cong \twomat{\epsilon}{\star}{0}{1}$.
    \end{itemize}
  \end{enumerate}
  Then $\Hc_{\mf}$ is free of rank $2$ over $\TT_{K, \mf}$.
\end{theorem}
\begin{proof} (sketch) For $v \in \Sigma$, let $K_v^{\min} \subset (D \otimes F_v)^\times$ be a maximal compact
  subgroup; otherwise, let $K_v^{\min} = K_v$.  Let $K^{\min} = \prod_{v}K_v^{\min}$. The numbered conditions were chosen to ensure that all the relevant local deformation rings corresponding to forms of level $K^{\min}$
  are formally smooth. Thus the Taylor--Wiles method gives a result analogous to \cite{MR1440309} Theorem~3.1 at level
  $K^{\min}$. The result at level $K$ now follows exactly as in the proof of \cite{MR1440309} Theorem~3.4, using Ihara's
  Lemma at each prime in $\Sigma$.  See also \cite{Taylor2006-MeromorphicContinuation} Theorem~3.2 for a similar result
  in the definite case.
\end{proof}
\begin{remarks}
  \begin{enumerate}
  \item In the `Fontaine--Laffaille case', at least if $(k-1)[F:\QQ] \leq l-2$, the version of Ihara's lemma required
    would presumably follow from the method of \cite{MR1262939}, as in \cite{Cheng}, and so the condition on the image
    of $\rhobar$ could be relaxed to a Taylor--Wiles hypothesis.  In the `ordinary case' we require Ihara's lemma at
    places of $\Sigma$ dividing $l$, which is apparently not accessible by the method of \cite{MR1262939}.
  \item Without a condition such as (5) where $D$ ramifies, the module may genuinely not be free,
    see~\cite{Manning}.
  \item Conditions (3) and (6) could probably be omitted, and the set $\Sigma$ of non-minimal places could probably be
    allowed to contain places where $\rhobar$ ramifies.
  \item The requirement that the weights are parallel is for convenience.  The restriction to the Fontaine--Laffaille
    range is not required for our version of Ihara's lemma, but is required to prove \emph{minimal} freeness results
    using the method of \cite{MR1440309}.  Nevertheless, in other situations where the multiplicity at minimal level can
    be determined (even if this multiplicity is not one), it seems plausible that Ihara's Lemma could be used to deduce
    information about the multiplicity at non-minimal levels.
  \end{enumerate}
\end{remarks}

\subsubsection{Local-global compatibility}
\label{sec:local-global}

In the work of Emerton \cite{emertonlocalglobal} on local-global compatibility in the $p$-adic Langlands progam, Ihara's
lemma is essential to obtain results with integral coefficients. Generalisations of Emerton's result to compact forms of
$U(2)$ over totally real fields in which $l$ splits have been proved in \cite{MR3652872} --- the compactness assumption
ensuring that Ihara's Lemma is known.  We expect that our results (and those of \cite{MR1262939}) could be used to prove
analogues of Emerton's Theorem~1.2.6 for the completed cohomology of Shimura curves, at least in settings where
multiplicity one still holds.
 
\subsubsection{Iwasawa theory}
\label{sec:bertolini-darmon}

In \cite{MR2178960}, Bertolini and Darmon proved one divisibility in the anticyclotomic Iwasawa Main Conjecture for
(certain) elliptic curves over imaginary quadratic fields.  The result of \cite{MR1262939} on Ihara's Lemma for Shimura
curves was an important technical tool in the proof.  Contingent on Ihara's Lemma for Shimura curves over totally real
fields, Longo \cite{MR2914851} generalises Bertolini and Darmon's work to the setting of Hilbert modular forms of
parallel weight two; our results therefore make his results unconditional in many cases.  Further generalisations are
made by Chida and Hsieh~\cite{MR3347993} and Wang~\cite{MR3487933}, and our work may be able to weaken some of their
hypotheses.

\subsubsection{Level raising}
\label{sec:level-raising}

The works \cite{MR804706} and \cite{MR1262939} apply Ihara's Lemma to the problem of level-raising for modular forms ---
that is, of determining at which non-minimal levels there is a newform with a given residual Galois representation.
Nowadays, there is an argument of Gee~\cite{MR2785764} using the Taylor--Wiles--Kisin method and a lifting technique of
Khare--Wintenberger.  Combined with the results of \cite{MR3324938} and of \cite{BLGGT2014-PotentialAutomorphy}, this
gives (under a Taylor--Wiles hypothesis) level raising theorems for Hilbert modular forms in arbitrary weight.  We thank
Toby Gee for explaining this point to us.  Since we also require the Taylor--Wiles hypothesis, it is unlikely that our
theorem gives substantial new level raising results.

\subsection{Outline of the paper}
\label{sec:outline}

In section~\ref{sec:shimura} we recall the definitions of Shimura curves and Hecke operators.  We also define the
Shimura sets we will need, and recall the necessary results on integral models.

Most of section~\ref{sec:deformation} is taken up with the calculation of local deformation rings at the auxiliary prime
$\qf$.  We also precisely define lattices in certain inertial types (representations of $GL_2(\Oc_{F, \qf})$).

Section~\ref{sec:patching} carries out the Taylor--Wiles--Kisin patching method.  We use the formalism of patching
functors, introduced in~\cite{MR3323575}.  This is mostly standard, and we include it because we don't know a reference for
the fact that the filtrations of homology coming from integral models may be patched.

Section~\ref{sec:commalg} contains calculations in commutative algebra over the local deformation rings at $\qf$ that
are at the technical heart of the proof.

Section~\ref{sec:ihara} contains the precise statement and proof of our theorem.

A sensible order to read this article in would be to skim section~\ref{sec:shimura}, to fix notation, and then turn to
section~\ref{sec:ihara}, referring back to the other sections as needed.

\subsection{Acknowledgments}
\label{sec:acknowledgements}

Firstly, we thank Matthew Emerton for suggesting that we collaborate on this project and for many enlightening
conversations.  We also thank Chuangxun Cheng, Fred Diamond, Toby Gee, Yongquan Hu, David Loeffler, Matteo Longo, and
Matteo Tamiozzo for useful comments or discussions.  Part of this work was written up while the second author was at the
Max Planck Institute for Mathematics, and he thanks them for their support.

\subsection{Notation}
\label{sec:notation}

If $k$ is a local or global field, then $G_k$ will denote its absolute Galois group.  If $l$ is a prime distinct from
the characteristic of $k$, then we write $\epsilon : G_k \rarrow \ZZ_l$ for the $l$-adic cyclotomic character and
$\bar{\epsilon}$ for its reduction modulo $l$.

If $l$ is a prime and $M$ is a $\ZZ_l$-module, then we write $M^\vee$ for its Pontrjagin dual.  If $M$ is a finite free
$\ZZ_l$-module (resp. an $\FF_l$-vector space, resp. a $\QQ_l$-vector space), then we write $M^* = \Hom_{\ZZ_l}(M,
\ZZ_l)$ (resp. $\Hom_{\FF_l}(M, \FF_l)$, resp. $\Hom_{\QQ_l}(M, \QQ_l)$).

\section{Shimura curves}
\label{sec:shimura}

\subsection{} Let $F$ be a totally real number field of degree $d$ and let $\Oc_F$ be the ring of integers of $F$.  We
write $\AA_{F,f}$ for the finite adeles of $F$.  If $v$ is a place of $F$ then we write $k_v$ for its residue field,
$\varpi_v$ for a fixed choice of uniformizer in $F_v$, and $\AA^{v}_{F,f}$ for the finite adeles of $F$ with the factor
$F_v$ dropped.  If $l$ is a rational prime then we write $\Sigma_l$ for the set of places of $F$ above $l$; we write
$\Sigma_\infty$ for the set of infinite places of $F$.

\subsection{}\label{sec:quaternion-algebra} Let $D$ be a quaternion division algebra over $F$ split at either no infinite places
(the \emph{definite case}) or exactly one infinite place, $\tau$ (the \emph{indefinite case}), and let $\Oc_D$ be a
maximal order in $D$.  We write $\Delta$ for the set of finite places of $F$ at which $D$ ramifies.  We assume that if
$F = \QQ$ and we are in the indefinite case then $\Delta$ is nonempty.

We write $G$ for the algebraic group over $\Oc_F$ associated to $\Oc_D^\times$, and $Z$ for its centre.

For every place $v$ at which $D$ splits we fix an isomorphism
$\kappa_v : \Oc_D \otimes_{\Oc_F} \Oc_{F,v} \isomto M_2(\Oc_{F,v})$.  We also denote by $\kappa_v$ the various
isomorphisms, such as $(D \otimes_F F_v)^\times \isomto GL_2(F_v)$, obtained from it. 

\subsection{} We fix a rational prime $l$ and a finite place $\pf$ of $F$ such that $\pf \not \in \Delta$; we do allow
the possibility that $\pf \mid l$.

\subsection{}\label{sec:U_0} Let $K$ be a compact open subgroup of $G(\AA_{F,f})$.  If $v$ is a finite place of $F$ then when it is
possible to do so we will write $K = K^vK_v$ for $K^v \subset G(\AA_{F,f}^v)$ and $K_v \subset G(F_v)$. A compact open
subgroup $K$ of $G(\AA_{F,f})$ is \emph{unramified} at $v$ if $v \not \in \Delta$ and $K = K^vG(\Oc_{F,v})$ for some
$K^v$, and that it is \emph{ramified} otherwise.  We let \[\Sigma(K) = \Delta \cup \{v : \text{ $K$ is ramified at $v$}\}.\]

If $v \not \in \Delta$ is a finite place of $F$, and $n \geq 1$, then we define $U_0(v^n)$ to be the subgroup
\[U_0(v^n) = \left\{\kappa_v^{-1}\twomat{a}{b}{c}{d}\in G(\Oc_{F,v}) : c\equiv 0\pmod{\varpi_v^n} \right\} \] of
$G(\Oc_{F,v})$, and 
\[U_1(v^n) = \left\{\kappa_v^{-1}\twomat{a}{b}{c}{d}\in U_0(v^n) : d\equiv 1\pmod{\varpi_v^n}\right\}. \]

If $K$ is unramified at $v$ then we write
\[K_0(v) = K^vU_0(v) \subset K = K^vG(\Oc_{F,v}).\]

\subsection{}\label{sec:complex}   
Suppose that we are in the indefinite case.  Letting $\Hc = \CC \setminus \RR$ be acted on by $GL_2(\RR)$ in the usual
way, via $\kappa_\tau$ we get an action of $G(F_\tau) \cong GL_2(\RR)$ on $\Hc$.  We say that $K$ is \emph{sufficiently
  small} if the action of $G(F) \cap gKg^{-1} / Z(F) \cap gKg^{-1}$ on $\Hc$ is free for every $g \in G(\AA_{F,f})$.  We
will assume throughout that all our compact open subgroups are sufficiently small.  We let
\[X_K(\CC) = G(\QQ) \backslash \left(G(\AA_{F,f})/K \times \Hc\right),\] a compact Riemann surface. By the theory of
Shimura varieties, there is a smooth projective curve $X_K$ over $F$ such that, when $F$ is considered as a subfield of
$\CC$ via $\tau$, the $\CC$-points of $X_K$ are given by the above formula.  For $\Fc$ a sheaf of abelian groups on
$X_K(\CC)$ we write $H^i(X_K, \Fc) = H^i(X_K(\CC), \Fc).$

\subsection{}\label{sec:action} Write $[\gamma,x]$ for the point in $X_K(\CC)$ corresponding to $\gamma \in G(\AA_{F,f})$ and $x \in \Hc$.
If $K' \subset K \subset G(\AA_{F,f})$ are compact open subgroups then there is a map $X_{K'} \rarrow X_K$ given on
complex points by $[\gamma, x] \mapsto [\gamma, x]$. For $g \in G(\AA_{F,f})$ there is a map
$\rho_g : X_K \rarrow X_{g^{-1}Kg}$ given on $\CC$-points by
\[\rho_g([\gamma, x]) = [\gamma g, x],\]
The maps $\rho_g$ define a right action of $G(\AA_{F,f})$ on the inverse system $(X_K)_K$; if $g^{-1}Kg \subset K'$ then
we will also write $\rho_g$ for the composite map
\[X_K \xrightarrow{\rho_g} X_{g^{-1}Kg} \rightarrow X_{K'}.\]

\subsection{}\label{sec:hecke} Let $M$ be an abelian group.  Suppose that $K_1, K_2 \subset G(\AA_{F,f})$ are sufficiently small and that
  $g \in G(\AA_{F,f})$.  Then, as in \cite{MR2730374} section~4, there are double coset operators
  \[[K_1gK_2] : H^i(X_{K_2}, M) \rarrow H^i(X_{K_1}, M)\] for $i = 0, 1, 2$.  
 If $v \not \in \Sigma(K)\cup \Sigma_\infty$ then we define the Hecke operators
  $T_v$ and $S_v$ to be the double coset operators \[T_v = \left[K \twomat{\varpi_v}{0}{0}{1} K\right]\] and
  \[S_v = \left[K\twomat{\varpi_v}{0}{0}{\varpi_v}K\right].\]  If $A$
  is a ring and $S$ is a finite set of places containing $\Delta \cup \Sigma_\infty$ then we write
\[\TT_A^S = A[T_v, S_v : v \not \in S],\] a polynomial ring in infinitely many variables which acts on
$H^i(X_{K}, M)$ for any $K$ for which $\Sigma(K) \subset S$ and any $A$-module $M$.

If $v\not \in \Delta$, then we define the Hecke operator $U_{v}$ to be the double coset operator
\[\left[K\twomat{\varpi_v}{0}{0}{1}K\right]\]
acting on any $H^i(K, M)$ for $M$ an abelian group (note that $U_v = T_v$ if $K$ is unramified at $v$).

\subsection{}\label{sec:definite} Now suppose that we are in the definite case.  A compact open subgroup $K \subset
G(\AA_{F,f})$ is \emph{sufficiently small} if, for every $g \in G(\AA_{F,f})$, we have 
\[G(F) \cap g^{-1}Kg \subset Z(F).\]
Again, we will always assume that our compact open subgroups are sufficiently small.  We define
\[Y_K = G(F) \backslash G(\AA_{F,f}) / K\] which is a finite set.  Exactly as in the indefinite case, we define an
action of $G(\AA_{F,f})$ on the inverse system $(Y_K)_K$, and actions of double coset operators $[K_1gK_2]$ and Hecke
operators $T_v$, $S_v$ and $U_v$ on the groups $H^0(Y_K, M)$, for any abelian group $M$. In particular, we obtain an
action of $\TT_A^S$ on $H^0(Y_K, M)$ for any finite set of places $S$ containing $\Sigma(K)$, ring $A$, and $A$-module
$M$.

\subsection{}\label{sec:eisenstein} Suppose that we are in the definite or indefinite case, and that $A$ is a finite
$\ZZ_l$-algebra, so that the residue field of any maximal ideal of $A$ is a finite
extension of $\FF_l$.
\begin{definition}
  A maximal ideal $\mf$ of $\TT_{A}^S$ is \emph{$G$-automorphic} of level $K$ if it is in the support of $H^i(X_K, A)$
  (in the indefinite case) or $H^i(Y_K, A)$ (in the definite case) for some $i$.  It is \emph{$G$-automorphic} if it is
  $G$-automorphic of level $K$ for some $K$.
\end{definition}

If $\mf$ is a  $G$-automorphic maximal ideal of $\TT_{A}^S$ then there is an associated semisimple representation
\[\rhobar_\mf : G_F \rarrow GL_2\left(\TT^S_{A}/\mf\right)\]
characterised by $\ch_{\rhobar_{\mf}(\Frob_v)}(X) = X^2 - T_vX + \Nm(v)S_v$ for all $v \not \in S \cup \Sigma_l$.

\begin{definition} An $G$-automorphic maximal ideal of $\TT_{A}^S$ is \emph{non-Eisenstein} if $\rhobar_{\mf}$ is
  absolutely irreducible, and \emph{Eisenstein} otherwise.  A $\TT_{A}^S$-module is Eisenstein if every maximal ideal in
  its support is Eisenstein.

  It is \emph{non-exceptional} if $\rhobar_{\mf}(G_F)$ contains a subgroup of $GL_2(\bar{\FF}_l)$ conjugate to
  $SL_2(\FF_l)$; equivalently if it is non-Eisenstein and the image of $\rhobar_{\mf}$ contains an element of order
  $l$.  Otherwise, it is \emph{exceptional}.
  \end{definition}

  \begin{proposition} \label{prop:02eisenstein} Suppose that we are in the indefinite case.  The $\TT^S_A$-modules
    $H^0(X_K, A)$ and $H^2(X_K, A)$ are Eisenstein.
  \end{proposition}
  \begin{proof}
    Let $\nu : G \rarrow \GG_{m,F}$ be the reduced norm.  There is (see \cite{MR860139} section~1.2) a bijection
    \[\pi_0(X_K(\CC)) \rarrow \AA_{F,f}^\times \big/ F^{\times,+}\nu(K)\]
    where $F^{\times,+}$ is the set of totally positive elements of $F^{\times}$.  Write $C_K$ for the group on the
    right.  If $g \in G(\AA_{F,f})$ then $C_K = C_{g^{-1}Kg}$ and the diagram
    \[ \begin{CD} \pi_0(X_K(\CC)) @>>> C_K \\ @V{\rho_g}VV @ V{\cdot \nu(g)}VV \\ \pi_0(X_{g^{-1}Kg}(\CC)) @>>>
        C_K \end{CD}\] commutes.  This implies that $\TT^S_{A}$ acts on $H^0(X_K(\CC), A) \cong A[C_K]$ via the
    homomorphism $\TT^S_{A} \rarrow A[C_K]$ given by
    \begin{align*} T_v &\mapsto (\Nm(v) + 1)[\varpi_v] \\ S_v &\mapsto [\varpi_v^2],\end{align*} where we write $[g]$
    for the basis element of $A[C_K]$ corresponding to $g$.  If $\nf$ is a maximal ideal of $A[C_K]$ with residue field
    $\FF$, corresponding to a character $\chi : C_K \rarrow \FF^\times$, then $T_v$ and $S_v$ act on $A[C_K]/\nf$ as
    $(\Nm(v) + 1)\chi(\varpi_v)$ and $\chi(\varpi_v^2)$ respectively.  If $\psi :G_F \rarrow \FF^\times$ is the
    character of $G_F$ associated to $\chi$ by class field theory, and $\rhobar = \psi \oplus \epsilon \psi$, then $T_v$
    and $\Nm(v)S_v$ act on $A[C_K]/\nf$ by the scalars $\tr(\rhobar(\Frob_v))$ and $\det(\rhobar(\Frob_v))$, so that the action
    of $\TT^S_A$ on $A[C_K]/\nf$ factors through an Eisenstein maximal ideal as required.  It follows that the action of
    $\TT^S_A $ on $H^0(X_K,A)$ is Eisenstein.

    The statement for $H^2$ follows from Poincar\'{e} duality
    \[H^2(X_K, A) \cong H^0(X_K, A^\vee)^\vee\] and the formulae $S_v^* = S_v^{-1}$ and
    $T_v^{*} = S_v^{-1}T_v$ for the adjoints of $T_v$ and $S_v$.
  \end{proof}

\subsection{}
\label{sec:coeffs}
Let $A$ be a finite $\ZZ_l$-algebra.  There is an exact functor $M \mapsto \Lc_M$ from the category of $A[K]$-modules on
which $K \cap Z(F)$ acts trivially, to the category of local systems of $A$-modules on $X_K(\CC)$ or $Y_K$.  If $S$ is a
finite set of places of $F$ containing $\Sigma(K) \cup \Sigma_l\cup \Delta \cup \Sigma_\infty$, and such that the action
of $K$ on $M$ factors through $\prod_{v \in S} K_v$, then we obtain an action of the Hecke algebra $\TT^S_A$ on each
cohomology group $H^i(X_K, \Lc_M)$ or $H^0(Y_K, \Lc_M)$.

\begin{proposition}\label{prop:eisenstein-coeffs} Suppose that we are in the indefinite case.  For any $A$, $M$ and $S$ as above, the
  $\TT^S_{A}$-module $H^i(X_K, \Lc_M)$ is Eisenstein for $i = 0, 2$.
\end{proposition}
\begin{proof}
  This is proved just as in Proposition~\ref{prop:02eisenstein}.
\end{proof}

\subsection{}\label{sec:degeneracy} Suppose that $K$ is unramified at $\pf$.  Let  $\omega = \twomat{\varpi_\pf}{0}{0}{1}$.  Then since $\omega K_0(\pf) \omega^{-1} \subset K$, we have two
degeneracy maps $\pi_1, \pi_2$ defined (in the notation of~\ref{sec:action}) by
\begin{align*}
  \pi_1 & = \rho_e : X_{K_0(\pf)} \rarrow X_K \\
  \pi_2 &= \rho_{\omega^{-1}} : X_{K_0(\pf)} \rarrow X_K
\end{align*}
(with similar formulae in the definite case).  If $A$ is an abelian group then we obtain maps
\[\pi_1^*, \pi_2^* : H^i(X_K, A) \rarrow H^i(X_{K_0(\pf)}, A)\]
with, again, similar formulae in the definite case.  We write
\[ \pi^* = \pi_1^* + \pi_2^* : H^i(X_K, A)^{2} \rarrow H^i(X_{K_0(\pf)}, A).\]
If $M$, $\Lc_M$, and $S$ are as in~\ref{sec:coeffs} and if $\pf \not \in \Sigma(K)\cup \Sigma_\infty$ is such that the action of $K$ on $M$
factors through $K^{\pf}$, then we can similarly define 
\[ \pi^* = \pi_1^* + \pi_2^* : H^i(X_K, \Lc_M)^{2} \rarrow H^i(X_{K_0(\pf)}, \Lc_M)\]
(and analogous maps in the definite case).

\subsection{}\label{sec:centre}
Define the finite abelian (class) group $\Gamma_K$ by
\[\Gamma_K = Z(\AA_{F,f})/Z(F) (K\cap Z(\AA_{F,f})).\]
It acts freely on $X_K$ and $Y_{K}$ by our assumption that $K$ is sufficiently small.

Suppose that $A$ is a finite $\ZZ_l$-algebra and that $\psi$ is a character $\AA_{F,f}^\times/F^\times \rarrow A^\times$
that vanishes on $K \cap Z(\AA_{F,f})$ (regarded as a subgroup of $Z(\AA_{F,f}) = \AA_{F,f}^\times$), so that we may
consider $\psi$ as a character of $\Gamma_K$.  For $M$ any $A[\Gamma_K]$-module, we write $M[\psi]$ for the largest
submodule of $M$ on which $\Gamma_K$ acts as $\psi$ and $M_\psi$ for the largest quotient module of $M$ on which
$\Gamma_K$ acts as $\psi$.

\begin{lemma}\label{lem:centre-injective}
  Let $A$ be as above, and let $\mf$ be a non-Eisenstein maximal ideal of $\TT^S_A$.  Then $H^0(Y_{K},A^\vee)_\mf$ and
  $H^1(X_K,A^\vee)_\mf$ are injective $A[\Gamma_K]$-modules.
\end{lemma}

\begin{proof}
  In the indefinite case, we use the Hochschild--Serre sequence and fact that $\mf$ is non-Eisenstein.  Let $V$ be an
  $A[\Gamma_K]$-module and let $\Lc_V^\vee$ be the local system on $X_K/\Gamma_K$ associated to $V^\vee$.  The action of the Hecke
  operators away from ramified primes descends to an action on $H^i(X_K/\Gamma_K, \Lc_V^\vee)$.  Then
  \[H^0(X_K/\Gamma_K, \Lc_V^\vee) = \Hom_{\Gamma_K}(H_0(X_K, A),V^\vee)\]
  is Eisenstein by Proposition~\ref{prop:eisenstein-coeffs}, and the same is true for $H^2(X_K/\Gamma_K, \Lc_V^\vee)$ by
  Poincar\'{e} duality.  As $H^0(X_K, \Lc_V^\vee)_\mf$
  vanishes by Proposition~\ref{prop:eisenstein-coeffs},
  \begin{align*}\Hom_{\Gamma_K}(V, H^1(X_K,A^\vee)_\mf) &= \Hom_{\Gamma_K}(H_1(X_K, A)_{\mf}, V^\vee) \\ &= H^1(X_K/\Gamma_K,
    \Lc_V^\vee)_\mf
  \end{align*} and the latter is an exact functor of $V$ as $\mf$ is non-Eisenstein. In the definite case the proof is
  similar but easier (and the assumption on $\mf$ is not actually necessary).
\end{proof}

\subsection{}\label{sec:integral}
For the rest of this section we suppose that we are in the indefinite case, and fix a finite place
$\qf \not \in \Delta \cup \Sigma_l$ of $F$, let $\Oc_{(\qf)}$ be the localization of $\Oc_F$ at $\qf$, let $k$ be the residue
field of $\qf$, and let $\bar{k}$ be an algebraic closure of $k$.  By a \emph{model} of $X_K$ we will mean a proper flat
$\Oc_{(\qf)}$-scheme $\Xf_K$ equipped with an isomorphism $\Xf_K \otimes_{\Oc_{(\qf)}} F \isomto X_K$.

We will consider $K$ that are (sufficiently small and) of the form $K^{\qf}GL_2(\Oc_{F,\qf})$ or $K^{\qf}U_0(\qf)$.  For
such $K$, there are models $\Xf_K$ of $X_K$ constructed by Morita \cite{MR625590} (in the first case) and by Jarvis
\cite{MR1669444}, following Carayol \cite{MR860139} (in the second).  They have the following properties:

\begin{theorem}\label{thm:integral} Suppose that $K$ is
  unramified at $\qf$.
  \begin{enumerate}
  \item The curve $\Xf_K$ is smooth over $\Oc_{(\qf)}$.
  \item \label{c} The curve $\Xf_{K_0(\qf)}$ is regular and $\Xf_{K_0(\qf)} \otimes_{\Oc_{(\qf)}} \bar{k}$ is the union
    of two curves, each isomorphic to $\Xf_{K} \otimes_{\Oc_{(\qf)}}\bar{k}$, that intersect transversely at a finite
    set of points.
 \end{enumerate}
\end{theorem}
\begin{remark} We will use implicitly the functoriality of these models.  For instance, if $K \subset K'$ are
  as above then the morphism $X_{K} \rarrow X_{K'}$ extends uniquely to a finite flat morphism between the
  models.  If $K_{\qf}$ is fixed, then the action of $G(\AA_{F,f}^{\qf})$ on the inverse system
  $(X_{K^{\qf}K_{\qf}})_{K^\qf}$ extends uniquely to the inverse system of models.  This action is compatible with
  varying $K_{\qf}$, and with the maps
  $\Xf_K \otimes_{\Oc_{(\qf)}} \bar{k} \rarrow \Xf_{K_0(\qf)} \otimes_{\Oc_{(\qf)}} \bar{k}$ implicit in part~\ref{c} of
  the theorem.
\end{remark}

\subsection{}\label{sec:ss} Suppose that $K$ is unramified at $\qf$.
  \begin{definition}
    The set of points where the two components of $\Xf_{K_0(\qf)} \otimes \bar{k}$ intersect maps injectively to
    $\Xf_{K}\otimes \bar{k}$ under the natural map $\Xf_{K_0(\qf)} \rarrow \Xf_{K}$.  The image is a finite set of
    points called the \emph{supersingular points} and is denoted $\Xf_K^{\text{ss}}$.  
  \end{definition}
  There is an adelic description of this set that we now explain.  Let $\bar{D}$ be a quaternion algebra over $F$
  ramified at $\Delta \cup \{\qf, \tau\}$ and let $\bar{G}$ be the algebraic group over $F$ associated to $\bar{D}^\times$.
  We fix a continuous isomorphism
  \[\iota : \bar{D} \otimes_F \AA^{\qf}_{F,f} \isomto D \otimes_F \AA_{F,f}^{\qf}.\]
  Let $\Oc_{\bar{D},\qf}$ be the unique maximal order of $D \otimes_F F_\qf$.  Then we write
  \[Y_{K^{\qf}} = Y_{\iota^{-1}(K^\qf)\Oc_{\bar{D}, \qf}}.\] 
\begin{remark}
  It follows from the Jacquet--Langlands correspondence that, if $K$ is unramified at $\qf$ and $\mf$ is in the support of
  $H^0(Y_{K^\qf}, A)$, then $\mf$ is in the support of $H^1(X_{K_0(\qf)}, A)$.
\end{remark}
  \begin{theorem}[\cite{MR860139}~(11.2)] \label{thm:ss} There is a $G(\AA_{F,f}^\qf)$-equivariant
    isomorphism of inverse systems
    \[(\Xf_{K}^{\text{ss}})_{K^{\qf}} \isomto (Y_{K^{\qf}})_{K^{\qf}}.\]
  \end{theorem}

\subsection{}\label{sec:vanishing} Suppose that $K$ is unramified at $\qf$ and that $\FF$ is a finite extension of
  $\FF_l$.   The geometry of $\Xf_{K_0(\qf)}$ and the theory of vanishing cycles allow us to relate
  $H^1(X_{K_0(\qf)}, \FF)$, $H^1(X_{K}, \FF)$ and $H^0(Y_{K^{\qf}}, \FF)$.  In the case at hand, this is worked out in
  \cite{MR1669444}, sections~14-18.  We recall the result in our notation:
  \begin{theorem} \label{thm:vanishing} Suppose that $K$ is unramified at $\qf$.  Let $S$ be a finite set of places
    containing $\Sigma(K) \cup \{\qf\} \cup \Sigma_\infty \cup \Delta$ and let $\mf$ be a non-Eisenstein
    maximal ideal of $\TT^S_K$.  Then there is a filtration
  \[0 \subset V_1 \subset V_2 \subset V = H^1(X_{K_0(\qf)}, \FF)_\mf\]
  together with isomorphisms
  \begin{align*}
  V_1 & \longisomto H^0(Y_{K^{\qf}}, \FF)_\mf, \\
  V_2/V_1 & \longisomto H^1(X_K,\FF)_\mf^{\oplus 2} 
  \intertext{and} 
  V/V_2 &\longisomto H^0(Y_{K^{\qf}}, \FF)_\mf.
  \end{align*}
  The filtration, and isomorphisms, are compatible with the transition morphisms for varying $K^{\qf}$ and with the
  action of the Hecke operators $T_v$ and $S_v$ for $v \not \in \Sigma(K) \cup \{\qf\}\cup \Delta$ and $U_v$ for
  $v\not \in \{\qf\}\cup \Delta$.
\end{theorem}
\begin{proof}
  As mentioned, this is proved in \cite{MR1669444}: we give references to that paper.  The key diagram is that at the
  end of section~14, which relates Hecke-modules $X(H)$, $Y(H)$, $\tilde{X}(H)$, $\tilde{Y}(H)$, $M(H)$, and $R(H)$.  In
  particular, there is a filtration of $M(H)$ with graded pieces $\tilde{X}(H)$, $R(H)$, and $\tilde{Y}(H)$.  Choosing
  the group $H$ in that paper appropriately, taking the sheaf there called $\Fc$ to be the constant sheaf $\FF$, and
  after localizing at $\mf$, we have that $M(H)_\mf$ is our $H^1(X_{K_0(\qf)}, \FF)_\mf$, while $R(H)_\mf$ is our
  $H^1(X_K,\FF)_\mf^{\oplus 2}$ (see \cite{MR1669444} Corollary~16.3).  A choice of ordering of the irreducible
  components of each connected component of the special fibre of $\Xf_{K_0(\qf)}$ gives, by Theorem~\ref{thm:ss}, an
  isomorphism between $Y(H)_\mf$ and $H^0(Y_{K^{\qf}}, \FF)_\mf$.  By Proposition~\ref{prop:02eisenstein}, or
  \cite{MR1669444} Lemma~18.1, we have $Y(H)_\mf \cong \tilde{Y}(H)_\mf$.  By \cite{MR1669444} Proposition~17.4 and
  Lemma~18.2, we have (Hecke-equivariant) isomorphisms $\tilde{X}(H)_\mf \cong X(H)_\mf \cong Y(H)_\mf$.  The result
  follows.
\end{proof}

It follows from Lemma~\ref{lem:centre-injective} that we can take $\psi$-parts in the filtration of
Theorem~\ref{thm:vanishing} to obtain a filtration of $H^1(X_{K_0(\qf)}, \FF)_\mf[\psi]$ with graded pieces
$H^0(Y_{K^{\qf}}, \FF)_\mf[\psi]$, $H^1(X_K,\FF)_\mf^{\oplus 2}[\psi]$, $H^0(Y_{K^{\qf}}, \FF)_\mf[\psi]$ for any
non-Eisenstein maximal ideal $\mf$ of $\TT^S_\FF$.

\section{Types and local deformation rings}
\label{sec:deformation}

For this section, let $L$ be a local field of characteristic $0$, with residue field $k$ of order $q$. Let $\Gamma$ be
the absolute Galois group of $L$, $I$ its inertia subgroup, and $P$ its wild inertia subgroup.  Let $\sigma \in I$ be a
lift of a topological generator of $I/P$, and let $\phi \in \Gamma$ be a lift of arithmetic Frobenius.  Then we have the
well-known relation $\phi\sigma\phi^{-1} = \sigma^q$ in $\Gamma/P$.

By a \emph{coefficient system} we will mean a triple $(E, \Oc, \FF)$ where: $E/\QQ_l$ be a finite extension, with ring
of integers $\Oc$, uniformizer $\varpi$, and residue field $\FF = \Oc/\varpi$. For now, we will take an arbitrary
coefficient system; later we will impose further conditions on $E/\QQ_l$.  

Let $\Cc_\Oc$ (resp. $\Cc_\Oc^\wedge$) be the category of Artinian (resp. complete Noetherian) local $\Oc$-algebras with
residue field $\FF$. We say that a functor $\Fc:\Cc_\Oc\to {\bf Set}$ is \emph{pro-represented} by some
$R\in \Cc^\wedge_\Oc$ if $\Fc$ is naturally isomorphic to $\Hom_{\Oc}(R,-)$.

Now fix a continuous representation $\rhobar:\Gamma\to GL_2(\FF)$. The primary goal of this section is to introduce
various deformation rings of $\rhobar$. Many treatments of this material assume that the coefficient ring $\Oc$ is
sufficiently ramified. For our purposes, it will be necessary to precisely control the ramification of $\Oc$, and so a
little more care will be needed in certain parts.

Consider the (framed) deformation functor $\Cc_\Oc\to {\bf Set}$
defined on objects $A$ by
\[
  A \mapsto \{\text{continuous lifts $\rho:\Gamma\to GL_2(A)$ of $\rhobar$}\} \]

It is well-known that this functor is pro-representable by some $R^\square_{\rhobar,\Oc}\in
\Cc^\wedge_\Oc$. Furthermore, $\rhobar$ admits a universal lift $\rho^\square:\Gamma\to GL_2(R^\square_{\rhobar,\Oc})$.

For any continuous homomorphism, $x:R^\square_{\rhobar,\Oc}\to \Ebar$, we obtain a Galois representation
$\rho_x:\Gamma\to GL_2(\Ebar)$ lifting $\rhobar$, from the composition
$\Gamma\xrightarrow{\rho^\square} GL_2(R^\square_{\rhobar,\Oc})\xrightarrow{x} GL_2(\Ebar)$.

For any character $\psi: \Gamma\to \Oc^\times$ with
$\det \rhobar\cong \psi \epsilon^{-1} \pmod{\varpi}$ define $R^{\square,\psi}_{\rhobar,\Oc}$ to be the quotient of
$R^\square_{\rhobar,\Oc}$ on which $\det \rho^{\square}= \psi \epsilon^{-1}$. Equivalently,
$R^{\square,\psi}_{\rhobar,\Oc}$ is the ring pro-representing the functor of lifts of $\rhobar$ with determinant
$\psi \epsilon^{-1}$.

As $l > 2$, we have an isomorphism
\begin{equation} \label{eq:twist}R^{\square, \psi}_{\rhobar, \Oc} \widehat{\otimes} R_{\det(\rhobar), \Oc} =
  R^\square_{\rhobar, \Oc}\end{equation} where $R_{\det(\rhobar), \Oc}$ is the universal deformation ring of the
character $\det(\rhobar)$.

\subsection{Deformation rings when $l\nmid q$}\label{sec:l!=p}
For this subsection, we assume that $l\nmid q$.  In this case, the $\Oc$-algebras $R^{\square,\psi}_{\rhobar,\Oc}$ and
$R^{\square}_{\rhobar,\Oc}$ are flat of relative dimensions 3 and 4, respectively.  The second statement follows from
\cite{shotton-gln}~Theorem~2.5.  The first statement follows from the second, the isomorphism~\eqref{eq:twist}, and the
flatness of the deformation ring of a character (see for example \cite{1702.06019app}~Lemma~2.5).  Shortly, we will
analyse these rings in more detail in a particular case.

\subsection{Deformation rings when $l|q$}\label{sec:l=p}

Now assume that $l|q$, so that $l$ is the residue characteristic of $L$.  If $L'/L$ is any finite extension, then
by~\cite{Kisin2008-PotentiallySemistable} there is a quotient $R^{\square, L'\text{-st}}_{\rhobar, \Oc}$ of
$R^\square_{\rhobar, \Oc}$ such that a continuous $\Oc$-algebra homomorphism
$x : R^{\square}_{\rhobar, \Oc} \rarrow \bar{E}$ factors through $R^{\square, L'\text{-st}}_{\rhobar, \Oc}$ if and only
if $\rho_x|_{G_{L'}}$ is semistable with parallel Hodge--Tate weights $\{0,1\}$.  For $\psi$ a finite order character of
$\Gamma$ that factors through $\Gal(L'/L)$, there is a quotient $R^{\square, \psi, L'\text{-st}}_{\rhobar, \Oc}$ of
$R^{\square, L'\text{-st}}_{\rhobar, \Oc}$ on which we additionally impose the condition
$\det(\rho) = \psi \epsilon^{-1}$.  We have that $\Spec R^{\square, \psi, L'\text{-st}}_{\rhobar, \Oc}$ is
equidimensional of dimension $3 + [L:\QQ_l]$.

\subsection{Deformation rings at the auxiliary prime $\qf$}\label{sec:ring-definitions}
In this subsection, we study the specific local deformation ring $R_{\qf} = R^{\square}_{\rhobar_\mf|_{F_\qf},\Oc}$
that will occur at the auxiliary prime $\qf$ in our argument, and define and compute certain quotients of it.

From now on assume that $q\equiv 1 \pmod{l}$ (so that in particular $l\nmid q$), and let
$\rhobar: \Gamma \rarrow GL_2(\FF)$ be the unramified representation with $\rhobar(\phi)=\twomat{1}{1}{0}{1}$.  Note
that both $\bar{\epsilon}$ and $\det(\rhobar)$ are the trivial character.

We will now impose a hypothesis on our coefficient system:
\begin{hypothesis} \label{hyp:coeffs} The coefficient system $(E, \Oc, \FF)$ is such that
  $\Oc = W(\FF)[\zeta + \zeta^{-1}]$ for a primitive $l$th root of unity $\zeta \in \Oc$.
\end{hypothesis}
Under this hypothesis, we write $W=W(\FF)$ be the ring of Witt vectors and let $E_0 = W[1/l]$, so that $E_0$ is an 
unramified extension of $\QQ_l$ with residue field $\FF$. We fix $\zeta \in \bar{E}_0$ a primitive $l$th root of
unity.  We also let
\[\pi = (\zeta - \zeta^{-1})^2 = (\zeta+\zeta^{-1})^2-4\in\Oc,\] 
and note that this is a uniformizer of $\Oc$.

We define the following quotients of $R^\square_{\rhobar, \Oc}$ in terms of the subfunctors that they represent:
\begin{itemize}
\item $R^{\nr}_{\rhobar, \Oc}$ parametrises lifts $\rho$ of $\rhobar$ that are unramified.
\item $R^{N}_{\rhobar, \Oc}$ parametrises lifts $\rho$ of $\rhobar$ such that
  \[\ch_{\rho(\sigma)}(T) = (T-1)^2\]
  and
  \[(\tr \rho(\phi))^2q = (q+1)^2 \det \rho(\phi).\]
\item $R^{\unip}_{\rhobar, \Oc}$ parametrises lifts $\rho$ of $\rhobar$ such that
  \[\ch_{\rho(\sigma)}(T) = (T-1)^2\] and 
  \[\left((\tr \rho(\phi))^2q - (q+1)^2 \det \rho(\phi)\right) \cdot (\rho(\sigma) - 1) = 0.\]
\item $R^{\ps}_{\rhobar, \Oc}$ parametrises lifts $\rho$ of $\rhobar$ such that
  \begin{align*}\ch_{\rho(\sigma)}(T) &= T^2 - (\zeta + \zeta^{-1})T + 1 \\ &= (T-  \zeta)(T - \zeta^{-1}).\end{align*}
\end{itemize}
\begin{remark}
  The relation ``$q\tr(\phi)^2 = (q + 1)^2 \det(\phi)$'' should be thought of as saying that the eigenvalues of $\rho(\phi)$
  are in the ratio $q : 1$, which is the case for all characteristic zero lifts of $\rhobar$ for which the image of
  inertia is non-trivial and unipotent.
\end{remark}

\begin{remark}
  It is important for us that $R^\ps_{\rhobar, \Oc}$ be defined over $\Oc$ and not just $\Oc[\zeta]$.
\end{remark}
Fix an unramified character $\psi : \Gamma \rarrow \Oc^\times$ lifting the trivial character $\det(\rhobar)\bar{\epsilon}$.
Note that, on each of these quotients, we have that $\det(\rho^\square)$ is unramified, and so agrees with $\psi\epsilon^{-1}$ on
$I$.  For $? \in \{N, \nr, \unip, \ps\}$, we make the following definitions:
\begin{itemize}
\item $R^{?, \psi}_{\rhobar, \Oc}$ is the quotient of $R^?_{\rhobar, \Oc}$ on which $\det(\rho^\square) = \psi \epsilon^{-1}$;
\item $\Rbar^?_{\rhobar} = R^?_{\rhobar,\Oc}\otimes_\Oc\FF$;
\item $\Rbar^{?, \psi}_{\rhobar} = R^{?, \psi}_{\rhobar, \Oc}\otimes_\Oc \FF$.
\end{itemize}

\subsection{}\label{sec:ring-calculations}We will need somewhat explicit descriptions of these rings, which were
obtained in Proposition~5.8 of \cite{MR3554238} and its proof.  Let
\[\rho^{\square}(\sigma) = 1 + \twomat{A}{B}{C}{D}\]
and
\[\rho^{\square}(\phi) = \twomat{1+P}{1+R}{S}{1+Q}.\]

We will choose more convenient coordinates.  We may replace $B$ by $X = \frac{B}{1+R}$, $Q$ by
$T = \tr(\rho^\square(\phi)) - 2$, and $S$ by $\delta = \det(\rho^{\square}(\phi)) - 1$.  By this we mean that the natural
map
\[\Oc[[A, X, C, D, P, T, R, \delta]] \rarrow R^{\square}_{\rhobar, \Oc}\]
is surjective, which follows from the formulae $B = (1+R)X$, $Q = T - P$, and $S = (1+R)^{-1}(T + P(T-P) - \delta)$.
Then we may replace $T$ by either
\[Y_1 = (\tr\rho^{\square}(\phi))^2 - 4 \det \rho^{\square}(\phi) \] or
\[Y_2 = (\tr\rho^{\square}(\phi))^2q - (q+1)^2 \det \rho^{\square}(\phi),\] by which we mean that the natural maps
\[\Oc[[A, X, C, D, P, R, \delta, Y_i]] \rarrow R^{\square}_{\rhobar, \Oc}\] are surjections.  This follows from the equation
$T = \sqrt{4 + Y_1 + 4\delta} - 2$ in the first case --- where the square root is defined by a convergent Taylor series, as
$l > 2$ --- and a similar expression in the second.  We have maps
  \[\alpha_i : \Oc[[X,Y_i, P, R, \delta]] \rarrow R^{\square}_{\rhobar, \Oc}.\]
  \begin{remark} \label{rem:fix-dets}
    Write $\gamma = \epsilon(\phi)^{-1}\psi(\phi) - 1 \in \Oc$.  Then the maps $\alpha_i$ descend to maps, also denoted $\alpha_i$,
    \[\alpha_i : \Oc[[X, Y_i, P, R]] \cong \Oc[[X, Y_i, P, R, \delta]]/(\delta - \gamma) \rarrow R^{\square, \psi}_{\rhobar,
        \Oc}.\] In the proofs of all of the following propositions we work without fixing determinants.  For each
    $? \in \{N, \nr, \unip, \ps\}$ we already have that $\det(\rho^\square)$ is unramified on the quotient
    $R^?_{\Oc, \rhobar}$.  This means that to get the fixed determinant versions in the statements, we simply quotient
    by $\delta - \gamma$.
  \end{remark}

\begin{proposition}\label{prop:ps-ring} The ring $R^{\ps, \psi}_{\rhobar, \Oc}$ is isomorphic (via $\alpha_1$) to
    \[\Oc[[X,Y_1,P,R]]/(X^2Y_1 - \pi).\]
    In particular, it is regular.
\end{proposition}
\begin{proof}
  This follows from the proof of \cite{MR3554238} Proposition~5.8 part 2.  The quantity denoted $y$ in the proof of that
  proposition is here equal to~1.  The variables $X_1, \ldots X_5$ in that proof are our variables $X, Y_1, P, R, 2P -
  T$, but by the above remarks we can replace $2P- T$ with $\delta$ and obtain that $\alpha_1$ defines an isomorphism
  \[\Oc[[X, Y_1, P, R, \delta]]/(X^2Y_1 - \pi)\cong R^{\ps}_{\rhobar, \Oc}.\]
  The result with fixed determinant follows.
\end{proof}
\begin{proposition}\label{prop:unip-ring}
  The ring $R^{\unip, \psi}_{\rhobar, \Oc}$ is isomorphic (via $\alpha_2$) to
    \[\Oc[[X,Y_2,P,R]]/(XY_2)\]
    and its quotients $R^{\nr, \psi}_{\rhobar, \Oc}$ and $R^{N, \psi}_{\rhobar, \Oc}$ are, respectively,
    \[\Oc[[X, Y_2, P, R]]/(X) \text{\quad and \quad} \Oc[[X, Y_2, P, R]]/(Y_2).\]
    In particular, these last two deformation rings are formally smooth.
\end{proposition}
\begin{proof}
  This is not quite in \cite{MR3554238} Proposition~5.8, as the quotient $R^{\unip}_{\rhobar, \Oc}$ is not considered
  there, but the method of proof extends easily --- we will be brief.  The proof shows that, if we write $U = P-Q$ and
  $\alpha(T) = \frac{(q-1)(2+T)}{q+1}$, then $R^{\unip}_{\rhobar, \Oc}$ is cut out of $\Oc[[A,X,C,D,U,T,R,S]]$ by the following equations:
  \begin{align*}
    A + D &= 0 \\
    A^2 + (1+R)XC &= 0 \\
    \star(4(1+R)S + (U^2 - \alpha(T)^2)) &= 0 \\
    A &= \frac{1}{2}X(U - \alpha(T)) \\
    2AS - C(U + \alpha(T)) &= 0 \\
    C &= A\alpha(T) + XS\\
    (q-1)(AU + (1+R)XS + (1+R)C) &= 0.
  \end{align*}
  Here $\star$ denotes each of $A, X, C, D$, so that the third line is really four equations. Note that the third
  line can be rewritten as $\star Y_2 = 0$.  The first, fourth and sixth lines show that $A, C$ and $D$ may be
  written in terms of $X, T, S$ and $U$.  Making these substitutions we see that this set of equations is equivalent to
  the single equation $X(4(1+R)S + (U^2 - \alpha(T)^2)) = 0$. But if we now replace $T$, $S$ and $U$ by $Y_2$, $\delta$
  and $P$ as discussed above, we obtain that $R^{\unip}_{\Oc, \rhobar}$ is the quotient of $\Oc[[X,Y_2, P, R, \delta]]$ by
  $XY_2 = 0$ as required.

  The expressions for the quotients $R^{\nr}_{\rhobar, \Oc}$ and $R^N_{\rhobar, \Oc}$ follow immediately, and finally we
  eliminate $\delta$ by imposing the fixed determinant condition.
\end{proof}
\begin{proposition}\label{prop:ring-comparison}The images of $Y_1$ and $Y_2$ are equal in $\Rbar^{\square, \psi}_{\rhobar}$.  Denoting this common image
    by $Y$, the diagram
    \[
    \begin{CD}
      \FF[[X,Y,P,R]]/(X^2Y) @>{\sim}>{\alpha_1}> \Rbar^{\ps, \psi}_{\rhobar} \\
      @VVV @VVV \\
      \FF[[X,Y, P,R]]/(XY) @>{\sim}>{\alpha_2}> \Rbar^{\unip,\psi}_{\rhobar}
    \end{CD}
    \]
    commutes.
\end{proposition}
\begin{proof}
  That the images of $Y_1$ and $Y_2$ are equal is immediate from $q \equiv 1 \pmod l$.  The diagram commutes since
  $\alpha_1$ and  $\alpha_2$ are equal as maps $\FF[[X,Y,P,R]] \rarrow \Rbar^{\square}_{\rhobar}$.
\end{proof}

\begin{remark}
  In \cite{MR3554238} it is assumed that $\zeta \in \Oc$, which is not the case for us --- however, this assumption is
  not used (only the assumption that $\zeta + \zeta^{-1} \in \Oc$, which is required to even define $R^{\ps}_{\Oc,
    \rhobar}$).
\end{remark}

\begin{remark}
  The proofs above show that each of our deformation rings $R^{?}_{\rhobar, \Oc}$ turns out to be reduced and
  $l$-torsion free, and therefore is one of the fixed-type deformation rings defined by a Zariski closure operation
  in~\cite{MR3554238}.
\end{remark}

\subsection{}\label{sec:types} Next we define various representations of $GL_2(\Oc_L)$ over $W$ (or extensions of $W$).  Let $\Gc = GL_2(k)$ and $\Bc$ be
its subgroup of upper triangular matrices.  We will always regard representations of $\Gc$ as representations of
$GL_2(\Oc_L)$ by inflation.  If $A$ is a ring, then we will write $\triv_A$ for $A$ with the trivial
action of any group under consideration.

Since $q + 1 = [\Gc : \Bc]$ is invertible in $W$, the natural map \[\triv_W \rarrow \Ind_{\Bc}^{\Gc}\triv_W\] splits, and so we 
define $\St_W$ by the formula
\[\Ind_{\Bc}^{\Gc}\triv_W = \triv_W \oplus \St_W.\] If $A$ is a $W$-algebra, then define $\St_A = \St_W \otimes_W A$; then we have
$\Ind_{\Bc}^{\Gc} \triv_A  = \triv_A \oplus \St_A$.

Now let $E_1 = E[\zeta]$ and $\chi :k^\times \rarrow E_1$ be a non-trivial character.  Let
$\chi \otimes \chi^{-1}: \Bc \rarrow E_1^\times$ be the character
\[(\chi \otimes \chi^{-1})\twomat{x}{z}{0}{y} = \chi(x) \chi^{-1}(y).\] Let
\[\sigma^{\ps}_{E_1} = \Ind_\Bc^\Gc(\chi \otimes \chi^{-1}).\] If $E = E_0[\zeta + \zeta^{-1}]$ as before then
$\sigma^{\ps}_{E_1}$ is isomorphic to its conjugate under the nontrivial element of $\Gal(E_1/E)$, which switches $\chi$
and $\chi^{-1}$.  It therefore has a model $\sigma_E^{\ps}$ over $E$, by the calculation of the Schur index of a
character of a finite general linear group in \cite{MR623680} Theorem~2a --- see also Lemma~3.1.1 of \cite{MR3323575}.
By \cite{MR3323575} Lemma~4.1.1, there is a unique $\Oc$-lattice $\sigma^{\ps}_\Oc$ in $\sigma_E^{\ps}$ such that there is a
nonsplit short exact sequence
\begin{equation} \label{ps-ses} 0 \rarrow \FF \rarrow \sigma_{\Oc}^{\ps} \otimes \FF \rarrow \St_\FF \rarrow 0.\end{equation}
For $A$ an $\Oc$-algebra, we let $\sigma_A^{\ps} = \sigma_{\Oc}^{\ps} \otimes_{\Oc}A$.

\subsection{The local Langlands correspondence.} \label{sec:local-global} Suppose first that we are in the setting of
section~\ref{sec:ring-definitions}.  For $\rho:G_L \rarrow GL_2(\bar{E}_0)$ a continuous representation, let $\pi(\rho)$
be the smooth admissible representation of $GL_2(L)$ associated to $\rho$ by the local Langlands correspondence, and let
$x : R^{\square}_{\rhobar, \Oc} \rarrow \bar{E}$ be the associated homomorphism.  Then we have:
\begin{proposition} \label{prop:local-global}
  \begin{enumerate}
  \item If $\pi(\rho)|_{GL_2(\Oc_L)}$ contains $\triv_{\bar{E}}$, then $x$ factors through $R^{\nr}_{\rhobar, \Oc}$.
  \item If $\pi(\rho)|_{GL_2(\Oc_L)}$ contains $\St_{\bar{E}}$, then $x$ factors through $R^{\unip}_{\rhobar, \Oc}$.
  \item If $\pi(\rho)$ is discrete series and $\pi(\rho)|_{GL_2(\Oc_L)}$ contains $\St_{\bar{E}}$, then $x$ factors through
    $R^{N}_{\rhobar, \Oc}$.
  \item If $\pi(\rho)|_{GL_2(\Oc_L)}$ contains $\sigma_{\bar{E}}^{\ps}$, then $x$ factors through $R^{\ps}_{\rhobar, \Oc}$.
  \end{enumerate}
\end{proposition}

Now suppose that we are in the setting of section~\ref{sec:l=p}.  Suppose that $D_L$ is a quaternion algebra over $L$
and $K$ is a compact open subgroup of $D_L$. If $\pi$ is an irreducible admissible representation of $D_L$ over $\bar{E}$,
then by the local Langlands and Jacquet--Langlands correspondences there is an associated Weil--Deligne representation
$(r_\pi, N_\pi)$.  We may and do choose a finite extension $L_K/L$ such that, for all $\pi$ having
a $K$-fixed vector, the restriction $r_\pi|_{G_{L_K}}$ is unramified.  It follows that, if $\pi$ has a $K$-fixed vector
and $\rho : G_L \rarrow GL_2(E)$ is a de Rham representation of parallel Hodge--Tate weights $\{0,1\}$ such that
$WD(\rho)^{ss} \cong (r_\pi, N_\pi)$, then $\rho |_{G_{L_K}}$ is semistable and so corresponds to a point of
$R^{\square, L_K\text{-st}}_{\rhobar, \Oc}$.  We write
\[R^{\square, \psi, K\text{-st}}_{\rhobar ,\Oc}\] for \[R^{\square, \psi, L_K\text{-st}}_{\rhobar, \Oc}.\] We will say
that a lift $\rho : \Gamma \rarrow GL_2(A)$ of $\rhobar$ is $K$-semistable if the associated map
$R^{\square}_{\rhobar, \Oc} \rarrow A$ factors through $R^{\square, K\text{-st}}_{\rhobar, \Oc}$.

\section{Patching}
\label{sec:patching}

The goal of this section is to summarize the Taylor--Wiles--Kisin patching construction, and to prove the results about it
that will be needed for the proof of Theorem~\ref{main thm}.  We choose a coefficient system $(E,\Oc,\FF)$, which we
will eventually require to satisfy Hypothesis~\ref{hyp:coeffs}.

\subsection{Ultrapatching}\label{sec:ultrapatching}
In this section we summarize the commutative algebra behind the patching method. For convenience we will use the
``ultrapatching'' construction introduced by Scholze in \cite{Scholze}; we follow closely the exposition of~\cite{Manning} section~4.

From now on, fix a nonprincipal ultrafilter $\uf$ on the natural numbers $\NN$ (it is well known that such an $\uf$ must
exist, provided we assume the axiom of choice). For convenience, we will say that a property $\Pc(n)$ holds for
\emph{$\uf$-many $i$} if there is some $I\in \uf$ such that $\Pc(i)$ holds for all $i\in I$.

For any sequence of sets $\As =\{A_n\}_{n\ge 1}$, we define their \emph{ultraproduct} to be the quotient
\[\uprod{\As} = \left(\prod_{n=1}^\infty A_n\right)/\sim\] where we define the equivalence relation $\sim$ by
$(a_n)_n\sim (a_n')_n$ if $a_i = a_i'$ for $\uf$-many $i$.

If the $A_n$'s are sets with an algebraic structure (eg. groups, rings, $R$-modules, $R$-algebras, etc.) then
$\uprod{\As}$ naturally inherits the same structure.

If each $A_n$ is a finite set, and the cardinalities of the $A_n$'s are bounded (this is the only situation we will
consider in this paper), then $\uprod{\As}$ is also a finite set and there are bijections
$\uprod{\As}\xrightarrow{\sim} A_i$ for $\uf$-many $i$. Moreover if the $A_n$'s are sets with an algebraic structure,
such that there are only finitely many distinct isomorphism classes appearing in $\{A_n\}_{n\ge 1}$ (which happens
automatically if the structure is defined by \emph{finitely} many operations, eg. groups, rings or $R$-modules or
$R$-algebras over a \emph{finite} ring $R$) then these bijections may be taken to be isomorphisms. This is merely
because our conditions imply that there is some $A$ such that $A\cong A_i$ for $\uf$-many $i$ and hence $\uprod{\As}$ is
isomorphic to the ``constant'' ultraproduct $\uprod{\{A\}_{n\ge 1}}$ which is easily seen to be isomorphic to $A$ if $A$
is a finite set.

Lastly, in the case when each $A_n$ is a module over a \emph{finite} local ring $R$, there is a simple algebraic
description of $\uprod{\As}$. Specifically, the ring $\ds\Rc = \prod_{n=1}^\infty R$ contains a unique maximal ideal
$\prF_\uf\in \Spec\Rc$ for which $\Rc_{\prF_\uf}\cong R$ and
$\ds\left(\prod_{n=1}^\infty A_n \right)_{\prF_\uf}\cong \uprod{\As}$ as $R$-modules. This shows that $\uprod{-}$ is a
particularly well-behaved functor in our situation. In particular, it is exact.

For the rest of this section, fix a power series ring $S_\infty = \Oc[[z_1,\ldots,z_d]]$ and consider the ideal
$\nf = (z_1,\ldots,z_d)$. Fix a sequence of ideals $\Ic_n\subseteq S_\infty$ such that for any open ideal
$\af\subseteq S_\infty$ we have $\Ic_n\subseteq\af$ for all but finitely many $n$.  Also define
$\Sbar_\infty = S_\infty/(\varpi) = \FF[[z_1,\ldots,z_d]]$ and
$\Icbar_n = (\Ic_n+(\varpi))/(\varpi)\subseteq \Sbar_\infty$.

For any finitely generated $S_\infty$-module $M$, we will say that the $S_\infty$-rank of $M$, denoted by
$\rank_{S_\infty}M$, is the cardinality of a minimal generating set for $M$ as an $S_\infty$-module.

We can now make our main definitions:

\begin{definition}\label{def:patching}
  Let $\Ms = \{M_n\}_{n\ge 1}$ be a sequence of finitely generated ${S_\infty}$-modules with
  $\Ic_n\subseteq \Ann_{S_\infty}M_n$ for all but finitely many $n$.
  \begin{itemize}
  \item We say that $\Ms$ is a \emph{weak patching system} if the
    ${S_\infty}$-ranks
    of the $M_n$'s are uniformly bounded. If we further have $\varpi M_n = 0$ for all but finitely many $n$, we say that
    $\Ms$ is a \emph{residual weak patching system}
  \item We say that $\Ms$ is a \emph{patching system} if it is a weak patching system, and we have
    $\Ann_{S_\infty}(M_n)=\Ic_n$ for all but finitely many $n$.  \item We say that $\Ms$ is a \emph{residual patching system} if it is a residual weak patching system, and we have
    $\Ann_{\Sbar_\infty}(M_n)=\Icbar_n$ for all but finitely many $n$.
  \item We say that $\Ms$ is MCM (resp. MCM residual)  
    if $\Ms$ is a patching system (resp. residual patching
    system) and $M_n$ is free over $S_\infty/\Ic_n$ (resp. $\Sbar_\infty/\Icbar_n$) for all but finitely many $n$.
  \end{itemize}
  Furthermore, assume that $\Rs = \{R_n\}_{n\ge1}$ is a sequence of finite local ${S_\infty}$-\emph{algebras}.
  \begin{itemize}
  \item We say that $\Rs = \{R_n\}_{n\ge1}$ is a \emph{(weak, residual) patching algebra}, if it is a (weak, residual)
    patching system.
  \item If $M_n$ is an $R_n$-module (viewed as an ${S_\infty}$-module via the ${S_\infty}$-algebra structure on $R_n$)
    for all $n$ we say that $\Ms = \{M_n\}_{n\ge1}$ is a \emph{(weak, residual) patching $\Rs$-module} if it is a (weak,
    residual) patching system.
  \end{itemize}
	
  Let $\wP$ be the category of weak patching systems, with the obvious notion of morphism. Note that this is naturally
  an abelian category.
	
  Now for any weak-patching system $\Ms$, we define its patched module to be the $S_\infty$-module
  \[\patch(\Ms) = \invlim_{\af}\uprod{\Ms/\af},\]
  where the inverse limit is taken over all open ideals of $S_\infty$. We may treat $\patch$ is as functor from $\wP$ to
  the category of $S_\infty$-modules.
	
  If $\Rs$ is a weak patching algebra and $\Ms$ is a weak patching $\Rs$-module, then $\patch(\Rs)$ inherits a natural
  $S_\infty$-algebra structure, and $\patch(\Ms)$ inherits a natural $\patch(\Rs)$-module structure.
\end{definition}

In the above definition, the ultraproduct essentially plays the role of the pigeonhole principal in the classical
Taylor--Wiles--Kisin construction, with the simplification that it is not necessary to explicitly define a ``patching
datum'' before making the construction. Indeed, if one were to define patching data for the $M_n/\af$'s (essentially,
imposing extra structure on each of the modules $M_n/\af$) then the machinery of ultraproducts would ensure that the
patching data for $\uprod{\Ms/\af}$ would agree with that of $M_n/\af$ for infinitely many $n$. It is thus easy to see
that our definition agrees with the classical construction (cf. \cite{Scholze}).

Thus the standard patching Lemmas (cf. \cite{Kisin2009-ModuliFFGSandModularity}, Proposition~3.3.1) can be rephrased as
follows:

\begin{proposition}\label{prop:patching}
  Let $\Rs$ be a weak patching algebra, and let $\Ms$ be an MCM patching $\Rs$-module. Then:
  \begin{enumerate}
  \item $\patch(\Rs)$ is a finite type $S_\infty$-algebra, and $\patch(\Ms)$ is a finitely generated \emph{free}
    $S_\infty$-module.
  \item The structure map $S_\infty\to \patch(\Rs)$ (defining the $S_\infty$-algebra structure) is injective, and thus
    $\dim \patch(\Rs) = \dim S_\infty$.
  \item The module $\patch(\Ms)$ is maximal Cohen--Macaulay over $\patch(\Rs)$, and $(\lambda,z_1,\ldots,z_d)$ is a regular
    sequence for $\patch(\Ms)$.
  \end{enumerate}
\end{proposition}

\begin{proposition}\label{prop:residual patching}
  Let $\Rs$ be a weak patching algebra, and let $\Msbar$ be an MCM residual patching $\Rs$-module. Then:
	\begin{enumerate}
        \item $\patch(\Rs)/(\varpi)$ is a finite type $\Sbar_\infty$-algebra, and $\patch(\Msbar)$ is a finitely generated
          \emph{free} $\Sbar_\infty$-module.
        \item The structure map $\Sbar_\infty\to \patch(\Rs)/(\varpi)$ is injective, and thus
          $\dim \patch(\Rs)/(\varpi) = \dim \Sbar_\infty$.
        \item The module $\patch(\Msbar)$ is maximal Cohen--Macaulay over $\patch(\Rs)/(\varpi)$, and $(z_1,\ldots,z_d)$ is
          a regular sequence for $\patch(\Msbar)$.
          
	\end{enumerate}
      \end{proposition}

\begin{proposition}\label{prop:mod reg seq}
  Let $\nf = (z_1,\ldots,z_d)\subseteq S_\infty$, as above. Let $R_0$ be a finite type $\Oc$-algebra, and let $M_0$ be a
  finitely generated $R_0$-module. If, for each $n\ge 1$, there are isomorphisms $R_n/\nf\cong R_0$ of $\Oc$-algebras
  and $M_n/\nf\cong M_0$ of $R_n/\nf\cong R_0$-modules, then we have $\patch(\Rs)/\nf\cong R_0$ as $\Oc$-algebras and
  $\patch(\Ms)/\nf\cong M_0$ as $\patch(\Rs)/\nf\cong R_0$-modules.
\end{proposition}

From the set up of Proposition~\ref{prop:patching} there is very little we can directly conclude about the ring
$\patch(\Rs)$. However in practice one generally takes the rings $R_n$ to be quotients of a fixed ring $R_\infty$ of the
same dimension as $S_\infty$ (and thus as $\patch(\Rs)$). Thus we define a \emph{cover} of a weak patching algebra
$\Rs = \{R_n\}_{n\ge 1}$ to be a pair $(R_\infty,\{\varphi_n\}_{n\ge 1})$, where $R_\infty$ is a complete, topologically
finitely generated $\Oc$-algebra of Krull dimension $\dim S_\infty$ and $\varphi_n:R_\infty\to R_n$ is a surjective
$\Oc$-algebra homomorphism for each $n$. It is straightforward to show the following (cf. \cite{Manning})

\begin{proposition}\label{prop:surjective cover}
  If $(R_\infty,\{\varphi_n\})$ is a cover of a weak patching algebra $\Rs$, then the $\varphi_n$'s induce a natural
  continuous surjection $\varphi_\infty:R_\infty\onto \patch(\Rs)$.
\end{proposition}

Combining this with Propositions~\ref{prop:patching} and~\ref{prop:residual patching} we get the following (using the
fact \cite[\href{https://stacks.math.columbia.edu/tag/0AAD}{Lemma 0AAD}]{stacks-project} that if $f:A\onto B$ is a
surjection of noetherian local rings, then a $B$-module $M$ is Cohen--Macaulay as an $A$-module if and only if
it is Cohen--Macaulay as a $B$-module):

\begin{corollary}\label{cor:max CM}
  Let $\Rs$ be a weak patching algebra and let $(R_\infty,\{\varphi_n\})$ be a cover of $\Rs$. If $\Ms$ is an MCM patching $\Rs$-module, then $\patch(\Ms)$ is a maximal Cohen--Macaulay $R_\infty$-module. If $\Msbar$ is an MCM residual patching $\Rs$-module, then $\patch(\Msbar)$ is a maximal Cohen--Macaulay $R_\infty/(\varpi)$-module.
\end{corollary}

In our arguments, it will be necessary to patch the filtration from Theorem~\ref{thm:vanishing}. This would certainly be
possible if $\patch$ were an exact functor. However, this is not true in general\footnote{For an easy counterexample,
  assume that $S_\infty/\Ic_n$ is $\varpi$-torsion free for all $n$ (a condition which will be satisfied for our choice
  of $\Ic_n$ below) and let $\Ms=\{S_\infty/\Ic_n\}_{n\ge 1}$. Define $\varphi=\{\varphi_n\}_{n\ge 1}:\Ms\to\Ms$ by
  $\varphi_n(x) = \varpi^nx$. Then $\varphi:\Ms\to \Ms$ is injective, $\patch(\Ms) = S_\infty$, and
  $\patch(\varphi):S_\infty\to S_\infty$ is the zero map.}, but we can prove a weaker statement which suffices for our
purposes:

\begin{lemma}\label{lem:patching exact}
  The functor $\patch(-)$ is right-exact. Moreover, if
  \[0\to \As\to \Bs\to \Cs\to 0\] is an exact sequence of weak patching systems then
  \[0\to \patch(\As)\to \patch(\Bs)\to \patch(\Cs)\to 0\] is exact, provided that either:
  \begin{itemize}
  \item $\Cs$ is MCM, or
  \item $\As$, $\Bs$ and $\Cs$ are all residual weak patching systems, and $\Cs$ is MCM residual.
  \end{itemize}
\end{lemma}
\begin{proof}
  Let $\Ab$ be the category of abelian groups. For any countable directed set $I$, let $\finAb^I$ be the category of
  inverse systems of \emph{finite} abelian groups indexed by $I$.  
	
  Now note that any $(A_i,f_{ji}:A_j\to A_i)\in \finAb^I$ clearly satisfies the Mittag-Leffler condition: For any
  $i\in I$ there is a $j\ge i$ for which $\im(f_{ki}) = \im(f_{ji})$ for all $k\ge j$ (since $A_i$ is finite, and
  $\{\im(f_{ji})\}_{j\ge i}$ is a decreasing sequence of subgroups). Thus by
  \cite[\href{https://stacks.math.columbia.edu/tag/0598}{Lemma 0598}]{stacks-project} it follows that
  $\invlim:\finAb^I\to\Ab$ is exact.
	
  Now assume that $\As$, $\Bs$ and $\Cs$ are weak patching systems, and that we have an exact sequence
  \[0\to \As\to \Bs\to \Cs\to 0\] Then for any $\af\subseteq {S_\infty}$, $\As/\af\to \Bs/\af\to \Cs/\af\to 0$ is
  exact, so by the exactness of $\uprod{-}$ we get the exact sequence
  \[\uprod{\As/\af}\to \uprod{\Bs/\af}\to \uprod{\Cs/\af}\to 0.\]
  Thus we have an exact sequence of inverse systems
  \[\big(\uprod{\As/\af}\big)_\af\to \big(\uprod{\Bs/\af}\big)_\af\to \big(\uprod{\Cs/\af}\big)_\af\to 0\]
  But now as $\uprod{\As/\af}$, $\uprod{\Bs/\af}$ and $\uprod{\Cs/\af}$ are all finite, and there are only countably many open ideals of $S_\infty$, the above argument shows that
  taking inverse limits preserves exactness, and so indeed
  \[\patch(\As)\to \patch(\Bs)\to \patch(\Cs)\to 0\]
  is exact. 
	
  Now assume that one of the further conditions of the lemma holds. Write $\As = \{A_n\}_{n\ge 1}$,
  $\Bs = \{B_n\}_{n\ge 1}$ and $\Cs = \{C_n\}_{n\ge 1}$. Then 
  letting $I_n = \Ann_{S_\infty}C_n$ (so that either $I_n = \Ic_n$ or $\Icbar_n$ for all $n\gg0$), we get that for all
  $n\gg0$,
  \[0\rarrow A_n\rarrow B_n\rarrow C_n\rarrow 0\] is an exact sequence of $S_\infty/I_n$-modules, and $C_n$ is a free
  $S_\infty/I_n$-module (this is true regardless of which case we are in). It follows that
  \[\Tor_1^{S_\infty/I_n}(C_n,S_\infty/\af) = 0\] for all $\af\subseteq S_\infty$, and so
  \[0\rarrow A_n/\af\rarrow B_n/\af\rarrow C_n/\af\rarrow 0\] is exact for all $n\gg0$. The same argument as above now
  shows that
  \[0\to \patch(\As)\to \patch(\Bs)\to \patch(\Cs)\to 0\] is exact.
\end{proof}

This now implies that $\patch$ preserves filtrations in the cases that will be relevant to us:

\begin{corollary}\label{cor:filtration}
  Let $\Vs$ be a residual weak patching system with a filtration
  \[0=\Vs^0\subseteq \Vs^1\subseteq \cdots\subseteq \Vs^r = \Vs\] by residual weak patching systems $\Vs^k$. For
  $k=1,\ldots,r$ let $\Ms^k = \Vs^k/\Vs^{k-1}$. Assume that
  the $\Ms^k$'s are all MCM residual. Then $\patch(\Vs)$ has a filtration
  \[0=\patch(\Vs^0)\subseteq \patch(\Vs^1)\subseteq \cdots\subseteq \patch(\Vs^r) = \patch(\Vs)\] with
  $\patch(\Vs^k)/\patch(\Vs^{k-1})\cong \patch(\Ms^k)$ for all $k=1,\ldots,r$.
\end{corollary}

One can also make an analogous statement about filtrations of weak patching systems, instead of residual weak patching
systems, but we will not need that result.

\begin{proof}
    For any $k\ge 1$ we have an exact sequence
  \[0\to \Vs^{k-1}\to \Vs^k\to \Ms^k\to 0.\] As $\Ms^k$ is MCM residual, Lemma~\ref{lem:patching exact} implies that
  the map $\patch(\Vs^{k-1}) \rarrow \patch(\Vs^{k})$ is an inclusion, and that
  $\patch(\Vs^k)/\patch(\Vs^{k-1})\cong \patch(\Ms^k)$. The result follows.
\end{proof}

\subsection{Global deformation rings}\label{sec:global}

We fix the following data:
\begin{itemize}
\item a quaternion division  algebra $D$ over $F$ split at exactly one infinite place, as in section~\ref{sec:shimura};
\item a coefficient system $(E, \Oc, \FF)$ satisfying Hypothesis~\ref{hyp:coeffs};
\item a non-Eisenstein maximal ideal $\mf\subseteq \TT^S_\Oc$ (for some set $S$, which we will not fix yet) which is $G$-automorphic;
\item a finite order character $\psi:G_F\to \Oc^\times$ for which $\psi \equiv \det \rhobar \epsilon \pmod\varpi$.  We
  also write $\psi$ for the character $\psi \circ \Art$, where $\Art : \AA_{F,f}^\times/F^\times \rarrow G_F^{\ab}$ is
the global Artin map.
\end{itemize}
Enlarging $\FF$ if necessary, we assume that the residue field of $\mf$ is $\FF$.  By definition, $\mf$ is
$G$-automorphic of some level $K_{\mf} \subset G(\AA_{F,f})$, which we fix temporarily.  Now we fix, for the rest of this
section:
\begin{itemize}
\item a finite place $\qf \not \in \Sigma_l \cup \Sigma(K_{\mf})$  of $F$ at which $\rhobar$ is unramified; 
\item a finite set $\Sigma$ of finite places of $F$ that contains $\Sigma_l \cup \{\qf\} \cup \Sigma(K_{\mf})$ (which means that we can, and will, regard $\mf$ as a maximal ideal of $\TT^\Sigma_\Oc$ rather than $\TT^S_\Oc$);
\item for each $v \in \Sigma_l$, a compact open subgroup $K^0_v \subset K_{\mf} \cap G(F_v)$.
\end{itemize}
We will use $S$ to denote a finite set of places of $F$.  In the following, $S$ and $K$ will sometimes vary but we will
always impose the following hypotheses on the pair $(S, K)$:
\begin{hypotheses}\label{hyp:K}\  
  \begin{itemize}
  \item $\mf$ is $G$-automorphic of level $K$;
  \item $S$ contains $\Sigma \cup \Sigma(K) \cup \Sigma_\infty$;
  \item $F^\times (K\cap Z(\AA_{F,f})) \subset \ker(\psi)$ (this implies that $\psi$ is unramified outside of $S$);
  \item for all $v \in \Sigma_l$, $K \cap G(F_v) \supset K^0_v$;
  \item $K$ has the form $K^\qf K_\qf$ for some $K^\qf \subset G(\AA_{F,f}^\qf)$ and $K_\qf \subset G(F_{\qf})$.
  \end{itemize}
\end{hypotheses}

Let $\rhobar = \rhobar_{\mf}:G_F\to GL_2(\FF)$, and note that $\rhobar$ is absolutely irreducible and unramified outside
of $S$. For any place $v$ of $F$, let $\rhobar_{v} = \rhobar|_{G_{F_v}}$. By taking a quadratic extension of $\FF$ if necessary, we will assume that for each $g\in G_F$, all of the eigenvalues of $\rhobar(g)$ lie in $\FF^\times$.

As in \cite[section 3.2]{Kisin2009-ModuliFFGSandModularity}, define $R^{\square}_{F,S}(\rhobar)\in \Cc_{\Oc}^\wedge$ to
be the $\Oc$-algebra pro-representing the functor $\Dc^{\square}_{F,S}(\rhobar):\Cc_\Oc\to {\bf Set}$ which sends $A$ to
the set of equivalence classes of tuples
\begin{equation}\label{eq:lift}\left(\rho, (\beta_v)_{v\in\Sigma} \right)\end{equation}
where:
\begin{itemize}
\item $\rho : G_{F,S} \rarrow GL_2(A)$ is a continuous lift of $\rhobar$;
\item for each $v\in \Sigma$, $\beta_v \in 1 + M_2(\mf_A)$ (we think of this as basis for $A^2$ lifting the standard basis of $\FF^2$);
\item for each $v \mid l$ the restriction $\rho \mid_{G_{F_v}}$ is $K^0_{v}$-semistable, in the notation of
  section~\ref{sec:local-global};
\item two such collections $(\rho, (\beta_v)_{v \in \Sigma})$ and $(\rho', (\beta'_v)_{v \in \Sigma})$ are equivalent if
  there is $\gamma \in 1 + M_2(\mf_A)$ such that $\rho' = \gamma \rho \gamma^{-1}$ and $\beta'_v = \gamma \beta_v$ for all $v\in\Sigma$.
\end{itemize}

Now let $\Dc^{\square,\psi}_{F,S}(\rhobar):\Cc_\Oc\to {\bf Set}$ be the subfunctor of $\Dc^{\square}_{F,S}(\rhobar)$
consisting of the tuples $\left (\rho, (\beta_v)_{v\in\Sigma} \right)$ with
$\det \rho = \psi\epsilon^{-1}$, and let $R^{\square,\psi}_{F,S}(\rhobar)\in \Cc_{\Oc}^\wedge$ be the $\Oc$-algebra
pro-representing $\Dc^{\square,\psi}_{F,S}(\rhobar)$.

Also define the \emph{unframed} deformation ring $R_{F,S}(\rhobar)$ to be the $\Oc$-algebra pro-representing the functor
$\Cc_\Oc\to {\bf Set}$ which sends $A$ to the set of equivalence classes of lifts $\rho : G_{F,S} \rarrow GL_2(A)$ such
that $\rho|_{G_{F_v}}$ is $K^0_{v}$-semistable for all $v \mid l$, two such lifts being equivalent if they are conjugate
by an element of $1 + M_2(\mf_A)$. Finally, define $R^{\psi}_{F,S}(\rhobar)$ to be the quotient of $R_{F,S}(\rhobar)$ on
which $\det\rho(g) = \psi(g)$ for all $g\in G_{F,S}$.  The unframed deformation rings $R_{F,S}(\rhobar)$ and
$R_{F,S}^\psi(\rhobar)$ exist because $\rhobar$ is absolutely irreducible. We will let
$\rho^{\univ}_{S}:G_{F,S}\to GL_2(R_{F,S}(\rhobar))$ be a representative for the universal equivalence class of lifts of
$\rhobar$, which induces a homomorphism $\rho^{\univ}_{S,\psi}:G_{F,S}\to GL_2(R_{F,S}^\psi(\rhobar))$.

There is a `forgetful' map $R^{\psi}_{F,S}(\rhobar) \to R^{\square,\psi}_{F,S}(\rhobar)$, which by
\cite[(3.4.11)]{Kisin2009-ModuliFFGSandModularity} is formally smooth of dimension $j = 4|\Sigma|-1$, and so we may identify $R^{\square,\psi}_{F,S} = R^\psi_{F,S}[[w_1,\ldots,w_j]]$.

Lastly, for any $v\in \Sigma$, let $R_v = R^{\square,\psi|_{G_{F_v}}}_{\rhobar|_{G_{F_v}},\Oc}$ if $v\nmid l$ and
$R_v = R^{\square, \psi, K^0_v\text{-st}}_{\rhobar ,\Oc}$ if $v|l$.  If $(\rho, (\beta_v)_{v\in \Sigma})$ is as in
equation~\eqref{eq:lift} then, for each $v \in \Sigma$, $\beta_v^{-1}\rho\beta_v$ is a lift of $\rhobar$ that only
depends on the equivalence class of $(\rho, (\beta_v)_{v \in \Sigma})$.  Restricting each $\beta_v^{-1}\rho \beta_v$ to
$G_{F_v}$ induces a map
\[\widehat{\otimes}_{v\in \Sigma}R_v \to  R^{\square,\psi}_{F,S}(\rhobar).\]
We write $R_{\loc}$ for $\displaystyle \widehat{\otimes}_{v\in \Sigma}R_v$.

The Taylor--Wiles--Kisin patching construction relies on carefully picking sets of auxiliary primes to add to the
level, using the following lemma (see \cite{Kisin2009-ModuliFFGSandModularity} Proposition~3.2.5).  

\begin{lemma}\label{lem:TW primes}
  Assume that $\rhobar$ satisfies the following conditions:
  \begin{enumerate}
  \item $\rhobar|_{G_{F(\zeta_l)}}$ is absolutely irreducible.
  \item If $l = 5$ and the image of the projective representation
    $\proj \rhobar_{\mf}:G_F\to GL_2(\FFbar_5)\onto PGL_2(\FFbar_5)$ is isomorphic to $PGL_2(\FF_5)$, then
    $\ker\proj \rhobar_{\mf}\not\subseteq G_{F(\zeta_5)}$. (This condition holds automatically whenever
    $\sqrt{5}\not\in F$.)
  \end{enumerate}
  Suppose that $S = \Sigma \cup \Sigma_\infty$. Then there exist integers $r,g\ge 0$ such
  that for each $n\ge 1$, there is a finite set $Q_n$ of primes of $F$ for which:
  \begin{itemize}
  \item $\# Q_n = r$.
  \item $Q_n\cap S = \emptyset$.
  \item For any $v\in Q_n$, $\Nm(v)\equiv 1\pmod{l^n}$.
  \item For any $v\in Q_n$, $\rhobar(\Frob_v)$ has two distinct eigenvalues in $\FF^\times$.
  \item There is a surjection $R_{\loc}[[x_1,\ldots,x_g]]\onto R^{\square,\psi}_{F,S\cup Q_n}(\rhobar)$ extending the
    map $R_{\loc}\to R^{\square,\psi}_{F,S\cup Q_n}(\rhobar)$.
	\end{itemize}
	Moreover, we have $\dim R_{\loc} = r+j-g+1$.
      \end{lemma}

      From now on, fix integers $r,g$ and a sequence $\Qc = \{Q_n\}_{n\ge 1}$ of sets of primes satisfying the
      conclusions of Lemma~\ref{lem:TW primes}. Define $R_n=R^{\psi}_{F,S\cup Q_n}(\rhobar)$,
      $R_n^\square=R^{\square,\psi}_{F,S\cup Q_n}(\rhobar)$ for $n\ge 1$
      and \[R_\infty=R_{\loc}[[x_1,\ldots,x_g]]=\widehat{\otimes}_{v\in \Sigma}R_v[[x_1,\ldots,x_g]],\] so that we have
      surjections $R_\infty\onto R_n^\square$ for all $n$. Also let $R_0=R^{\psi}_{F,S}(\rhobar)$ and
      $R_0^{\square}=R^{\square,\psi}_{F,S}(\rhobar)$. Note that $R_n^\square\cong R_n[[w_1,\ldots,w_j]]$ for all
      $n\ge 0$ and $\dim R_\infty = r+j+1$.
\subsection{Patched modules over Shimura curves and sets}\label{sec:M_infty}

As before, we use $S$ to denote a finite set of places of $F$ containing $\Sigma \cup \Sigma_\infty$, and $K$ to denote a
compact open subgroup of $G(\AA_{F,f})$, such that $S$ and $K$ satisfy Hypotheses~\ref{hyp:K}.  In particular, there is
a maximal ideal $\mf$ of $\TT^S_{\Oc}$ that is $G$-automorphic of level $K$.  Let $\TT(K,S)$ denote the image
of $\TT^S_{\Oc,\mf}$ in $\End_{\Oc}(H^1(X_{K},\Oc)_{\mf}[\psi])$.  Then $\TT(K,S)$ is a finite rank free $\Oc$-algebra
which is local with maximal ideal $\mf$. Note that $\TT(K,S)$ depends on the choices of $\mf$ and $\psi$ but we suppress these from
the notation.

As in section~6 of \cite{MR3323575} we have the following:
\begin{lemma}\label{lem:R->T}For any
  compact open $K$ and set $S$ as above, there exists a natural surjection $R^\psi_{F,S}(\rhobar)\onto \TT(K,S)$ with the
  property that $\rho^{\univ}_{S,\psi}(\tr(\Frob_v))\mapsto T_v$ and
  $\rho^{\univ}_{S,\psi}(\det(\Frob_v))\mapsto \Nm(v)S_v$ for any $v\not\in S$. These maps are compatible with the
  restriction maps $\TT(K',S')\to \TT(K,S)$ for $K'\subseteq K$ and $S \subset S'$.
\end{lemma}

If $S \subset S'$ are sets as above, then by Lemma~\ref{lem:R->T} and the definitions we have a commutative diagram
\[
\begin{CD}
{R^\psi_{F,S'}} @>>>  {\TT(K,S')} \\
@VVV @VVV \\
{R^\psi_{F,S}} @>>>{\TT(K,S)}                
\end{CD}
\]
where the left hand vertical map and the horizontal maps are surjections. It follows that the right hand vertical map,
injective by definition, is an isomorphism.  We therefore drop $S$ from the notation and write
\[\TT_K = \TT(K, S)\]
for any $K$ and $S$ satisfying Hypotheses~\ref{hyp:K}.

These Hecke algebras also act on the spaces $H^0(Y_{K^{\qf}},\Oc)_\mf[\psi]$, by the following lemma.

\begin{lemma}\label{lem:JL-map}
  For any compact open $K$ and set $S$ as above such that $K$ is unramified at $\qf$, the map
  $\TT_{\Oc, \mf}^S \rarrow \End(H^0(Y_{K^\qf},\Oc)_\mf[\psi])$ factors through the quotient $\TT_{K_0(\qf)}$.
\end{lemma}

\begin{proof}
  As $H^0(Y_{K^\qf},\Oc)_\mf[\psi]$ is torsion-free, we may check this after inverting $l$. It is then a consequence of
  the Jacquet--Langlands correspondence and the semisimplicity of $H^0(Y_{K^\qf},\Oc)_\mf[\psi]$ as a module over
  $\TT^S_{\Oc,\mf}$.
\end{proof}

We now fix $S$ to be the union of $\Sigma\cup\Sigma_\infty$, and let $\Qc = \{Q_n\}_{n\ge 1}$ be the
sequence of sets of places provided by Lemma~\ref{lem:TW primes}. For any $n\ge 1$, let $\Delta_n$ be the maximal
$l$-power quotient of $\ds\prod_{v\in Q_n}k_v^\times$. Consider the ring $\Lambda_n=\Oc[\Delta_n]$, and note
that:
\[\Lambda_n\cong\frac{\Oc[[y_1,\ldots,y_r]]}{\left((1+y_1)^{l^{e(n,1)}}-1,\ldots, (1+y_r)^{l^{e(n,r)}}-1\right)}\]
where $l^{e(n,i)}$ is the $l$-part of $\Nm(v)-1 = \# k_v^\times$, so that $e(n,i)\ge n$ by assumption. Let
$\af_n = (y_1,\ldots,y_r)\subseteq \Lambda_n$ be the augmentation ideal. Also define
\[\Lambdabar_n = \Lambda_n\otimes\FF\cong \frac{\FF[[y_1,\ldots,y_r]]}{\left((1+y_1)^{l^{e(n,1)}}-1,\ldots,
      (1+y_r)^{l^{e(n,r)}}-1\right)} = \frac{\FF[[y_1,\ldots,y_r]]}{\left(y_1^{l^{e(n,1)}},\ldots,
      y_r^{l^{e(n,r)}}\right)}\]

Now let $\ds H_n = \ker\left(\prod_{v\in Q_n}k_v^\times\onto\Delta_n\right)$.  For any finite place $v$ of $F$, there is
a group homomorphism $U_0(v)\to k_v^\times$ given by $\begin{pmatrix}a&b\\c&d\end{pmatrix}\mapsto ad^{-1}\pmod{v}$. Now
let $\ds U_H(Q_n)\subseteq \prod_{v\in Q_n}U_0(v)$ be the preimage of $\ds H_n\subseteq \prod_{v\in Q_n}k_v^\times$
under the map
\[\prod_{v\in Q_n}U_0(v)\onto  \ds\prod_{v\in Q_n}k_v^\times\]
Finally, for any $K$ (satisfying~\ref{hyp:K} for the set $S$), let $K_n$ be the preimage of $U_H(Q_n)$ under
\[K\into G(\AA_{F,f})\onto \prod_{v\in Q_n}G(F_v).\] We also let $K_0 = K$, and remark that for $n \geq 1$, $K_n$
and $S \cup Q_n$ satisfy~\ref{hyp:K}; in particular, $K_n = K_n^\qf K_\qf$.  For any $n\ge 0$, let
$\TT_{n,K} = \TT_{K_n}$.

Now for any $n\ge 1$ consider the $\Oc$-algebra
\[\TT_{n,K}\left[U_v\right]_{v\in Q_n}\subseteq \End_{\Oc}(H^1(X_{K_n},\Oc)_{\mf}[\psi]).\]
Now for each $v\in Q_n$ fix a choice $\alpha_v\in \FF^\times$ of eigenvalue for $\rhobar(\Frob_v)$ (recall that by
assumption, for each $v\in Q_n$ $\rhobar(\Frob_v)$ has two distinct eigenvalues in $\FF^\times$, and so there are
$2^{|Q_n|}$ ways to pick the system $(\alpha_v)_{v\in Q_n}$). Now define the ideal
\[\mftilde_n = \left(\mf,U_v-\alpha_v\right)\subseteq \TT_{n,K}\left[U_v\right]_{v\in Q_n}.\]

Now for each $n\ge 1$, define $\TTtilde_{n,K} = \left(\TT_{n,K}\left[U_v\right]_{v\in Q_n}\right)_{\mftilde_n}$. Also
define $\TTtilde_{0,K} = \TT_{0,K}$ and $\mftilde_0 = \mf$.

As in \cite[section 2]{Taylor2006-MeromorphicContinuation} we have:
\begin{lemma}\label{lem:R->Ttilde}
  The ring $\TTtilde_{n,K}$ is a finite $\TT_{n,K}$-algebra and $\mftilde_n$ is a
  maximal ideal of it lying over $\mf$.  The composite map
  \[R_n\rarrow \TT_{n,K} \rarrow \TTtilde_{n,K}\] is surjective.
  Moreover, there exist $\Oc$-algebra maps $\Lambda_n\rarrow R_n$ and
  $\Lambda_n\rarrow \TTtilde_{n,K}$ making the above map a surjection of
  $\Lambda_n$-algebras.
\end{lemma}
By definition, $\TT_{n,K}\left[U_v\right]_{v\in Q_n}$ acts on $H^1(X_{K_n},\Oc)_{\mf}[\psi]$ and
$ H^1(X_{K_n},\FF)_{\mf}[\psi]$ (the latter through its quotient
$\TT_{n,K}\left[U_v\right]_{v\in Q_n}\otimes_{\Oc}\FF$). Also, by Theorem~\ref{thm:vanishing}, if $K$ is unramified at
$\qf$ then $\TT_{n,K_0(\qf)}\left[U_v\right]_{v\in Q_n}\otimes_{\Oc}\FF$ acts on $H^0(Y_{K^\qf_n},\FF)_{\mf}[\psi]$.

So now for any $n\ge 0$ we can define
\begin{align*}
  M_{n,K} &= H_1(X_{K_n},\Oc)_{\mftilde_n, \psi^{-1}} = H^1(X_{K_n},\Oc)_{\mftilde_n}[\psi]^*,\\
  \Mbar_{n,K} &= M_{n,K}\otimes\FF = H_1(X_{K_n},\FF)_{\mftilde_n, \psi^{-1}} = H^1(X_{K_n},\FF)_{\mftilde_n}[\psi]^*,
\end{align*} 
and 
\begin{align*}
  N_{n,K^\qf} &= H_0(Y_{K^\qf_n},\Oc)_{\mftilde_n, \psi^{-1}} = H^0(Y_{K^\qf_n},\Oc)_{\mftilde_n}[\psi]^*, \\
  \Nbar_{n,K^\qf} &= H_0(Y_{K^\qf_n},\FF)_{\mftilde_n, \psi^{-1}} = H^0(Y_{K^\qf_n},\FF)_{\mftilde_n}[\psi]^*.
\end{align*}

The reason for dualizing is that the patching argument works more naturally with homology rather than
cohomology.

Note that $M_{n,K}$ and $\Mbar_{n,K}$ are naturally $\TT_{n,K}$-modules and, if $K$ is unramified at $\qf$, then
$N_{n, K^{\qf}}$ and $\Nbar_{n,K^\qf}$ are naturally $\TT_{n,K_0(\qf)}$-modules by Lemma~\ref{lem:JL-map}. In
particular we may regard them all as $R_n$-modules.

We now have the following result, a standard ingredient in the patching argument (see for instance
\cite{Kisin2009-ModuliFFGSandModularity}, \cite{MR3294620}, and \cite{MR3323575}):

\begin{proposition}\label{prop:aug ideal}
  For any $n\ge 1$ and any $K$, the map $\Lambda_n\rarrow R_n$ from Lemma~\ref{lem:R->Ttilde} makes $M_{n,K}$ and
  $N_{n,K^{\qf}}$ into finite rank free $\Lambda_n$-modules.  In particular, the maps $\Lambda_n \rarrow R_n$ and
  $\Lambda_n \rarrow \TTtilde_{n,K}$ are injective.  Moreover, the natural maps define an isomorphism
  $R_n/\af_n\cong R_0$ and isomorphisms $M_{n,K}/\af_n\cong M_{0,K}$ and $N_{n,K^{\qf}}/\af_n\cong N_{0,K^{\qf}}$ of
  $R_0$-modules.

  Similarly $\Mbar_{n,K}$ and $\Nbar_{n,K^{\qf}}$ are finite rank free $\Lambdabar_n$-modules and we have
  $\Mbar_{n,K}/\af_n\cong \Mbar_{0,K}$ and $\Nbar_{n,K^{\qf}}/\af_n\cong \Nbar_{0,K^{\qf}}$.

  In particular, $\rank_{\Lambda_n}R_n = \rank_\Oc R_0$,
  \[\rank_{\Lambda_n}M_{n,K} = \rank_{\Lambdabar_n}\Mbar_{n,K} = \rank_\Oc M_{0,K},\]
  and 
    \[\rank_{\Lambda_n}N_{n,K^{\qf}} = \rank_{\Lambdabar_n}\Nbar_{n,K^{\qf}}=\rank_\FF \Nbar_{0,K^{\qf}}\]
  for all $n\ge 1$, and so these ranks are independent of $n$.
\end{proposition}

We can now define framed versions of all of these objects. First let
\begin{align*}
  \Lambda_n^\square &= \Lambda_n[[w_1,\ldots,w_j]] \cong \frac{\Oc[[y_1,\ldots,y_r,w_1,\ldots,w_j]]}{\left((1+y_1)^{l^{e(n,1)}}-1,\ldots, (1+y_r)^{l^{e(n,r)}}-1\right)}\\
  \Lambdabar_n^\square &= \Lambdabar_n[[w_1,\ldots,w_j]] \cong \frac{\FF[[y_1,\ldots,y_r,w_1,\ldots,w_j]]}{\left(y_1^{l^{e(n,1)}},\ldots, y_r^{l^{e(n,r)}}\right)}
\end{align*}

Now define
\begin{align*}
  M_{n,K}^\square &= M_{n,K}\otimes_{R_n}R_n^{\square} = M_{n,K}\otimes_{\Lambda_n}\Lambda_n^{\square} = M_{n,K}[[w_1,\ldots,w_j]]
\end{align*}
and define $\Mbar_{n,K}^\square$, $N_{n, K^{\qf}}^\square$, and $\Nbar_{n,K^{\qf}}^\square$ similarly. Also note that
$R_n^\square\cong R_n\otimes_{\Lambda_n}\Lambda_n^\square\cong R_n[[w_1,\ldots,w_r]]$.

Now let $S_\infty =\Oc[[y_1,\ldots,y_r,w_1,\ldots,w_j]]$ and consider the ideals
\[\Ic_n = \left((1+y_1)^{l^{e(n,1)}}-1,\ldots, (1+y_r)^{l^{e(n,r)}}-1\right)\subseteq S_\infty.\]
Note that:
\begin{lemma}
  For any open ideal $\af\subseteq S_\infty$, we have $\Ic_n\subseteq \af$ for all but finitely many $n$.
\end{lemma}
\begin{proof}
  As ${S_\infty}/\af$ is finite, and the group $1+\mf_{S_\infty}$ is pro-$l$, the group
  $(1+m_{S_\infty})/\af = \im(1+m_{S_\infty}\into {S_\infty}\onto {S_\infty}/\af)$ is a finite $l$-group. Since
  $1+y_i\in 1+m_{S_\infty}$ for all $i$, there is an integer $k\ge 0$ such that $(1+y_i)^{\ell^k}\equiv 1 \pmod{\af}$
  for all $i=1,\ldots,r$. Then for any $n\ge k$, $e(n,i)\ge n\ge k$ for all $i$, and so indeed $\Ic_n\subseteq \af$ by
  definition.
\end{proof}
Thus we may apply the results of section~\ref{sec:ultrapatching} with this ring $S_\infty$ and these ideals
$\Ic_n$. Note that \[\dim S_\infty = 1+r+j = \dim R_\infty.\]

Let $\nf=(y_1,\ldots,y_r,w_1,\ldots,w_j)\subseteq S_\infty$, and identify $\Lambda_n^\square$ with $S_\infty/\Ic_n$ via the
above isomorphism.

Tensoring everything in Proposition~\ref{prop:aug ideal} with $\Lambda_n^\square$, we get that $M_{n,K}^\square$ is free
of rank $\rank_\Oc M_{0,K}$ over $\Lambda_n^\square$ for all $n$ with
$M_{n,K}^\square/\nf\cong M_{n,K}/\af_n\cong M_{0,K}$. Similar statements hold for $\Mbar_{n,K}^\square$,
$N_{n,K^{\qf}}^\square$, and $\Nbar_{n,K^{\qf}}^\square$.

Summarizing the results of this section in the language of section~\ref{sec:ultrapatching}, we have:

\begin{proposition}
  The sequence $\Rs^\square=\{R_n^\square\}_{n\ge1}$ is a patching algebra and $R_\infty$ is a cover of
  $\Rs^{\square}$. The sequences
\begin{align*}
\Ms_K^\square &= \{M_{n,K}^\square\}_{n\ge1}&
&\text{and}
&
\Ns_{K^{\qf}}^\square &= \{N_{n,K^{\qf}}^\square\}_{n\ge1}
\end{align*}
are MCM patching $\Rs^\square$-modules, and the sequences
\begin{align*}
\Msbar_K^\square &= \{\Mbar_{n,K}^\square\}_{n\ge1}&
&\text{and}
&
\Nsbar_{K^{\qf}}^\square &= \{\Nbar_{n,K^{\qf}}^\square\}_{n\ge1}
\end{align*}
are MCM residual patching $\Rs^\square$-modules.

For all $n\ge 1$ we have $R_n^\square/\nf\cong R_0$ and $M_{n,K}^\square/\nf\cong M_{0,K}$,
$\Mbar_{n,K}^\square/\nf\cong \Mbar_{0,K}$, $N_{n,K^{\qf}}^\square/\nf\cong N_{0,K^{\qf}}$ and
$\Nbar_{n,K^{\qf}}^\square/\nf\cong \Nbar_{0,K^{\qf}}$ as $R_0$-modules.
\end{proposition}

So now define the patched modules:
\begin{align*}
M_{\infty,K}&=\patch(\Ms_K^\square),\\
\Mbar_{\infty,K}&=\patch(\Msbar_K^\square),\\
N_{\infty,K^{\qf}}&=\patch(\Ns_{K^{\qf}}^\square), \quad \text{and} \\
\Nbar_{\infty,K^{\qf}}&=\patch(\Nsbar_{K^{\qf}}^\square).
\end{align*}
All of these modules are technically framed objects but, following standard convention, we are suppressing the $\square$
in our notation.

By Corollary~\ref{cor:max CM} it follows that $M_{\infty,K}$ and $N_{\infty,K^{\qf}}$ are maximal Cohen--Macaulay
$R_\infty$-modules, and $\Mbar_{\infty,K}$ and $\Nbar_{\infty,K^{\qf}}$ are maximal Cohen--Macaulay
$\Rbar_\infty=R_\infty/(\varpi)$-modules.

Moreover, Proposition~\ref{prop:mod reg seq} gives that $M_{\infty,K}/\nf\cong M_{0,K}$,
$\Mbar_{\infty,K}/\nf\cong \Mbar_{0,K}$, $N_{n,K^{\qf}}^\square/\nf\cong N_{0,K^{\qf}}$, and
$\Nbar_{\infty,K^{\qf}}/\nf\cong \Nbar_{0,K^{\qf}}$, as $R_0$-modules.

Now consider the filtration from Theorem~\ref{thm:vanishing}. By dualizing this, completing at $\mf$, and applying
$-\otimes_{\Lambda_n}\Lambda_n^\square$ we get a filtration
\[0=V_0\subseteq V_1\subseteq V_2 \subseteq V_3 = \Mbar_{n,K_0(\qf)}^\square\]
of $R_n^\square$-modules, with isomorphisms
\begin{align*}
	V_1 & \longisomto \Nbar_{n,K^{\qf}}^\square, \\
	V_2/V_1 & \longisomto (\Mbar_{n,K}^\square)^{\oplus 2},
	\intertext{and} 
	V_3/V_2 &\longisomto \Nbar_{n,K^{\qf}}^\square
\end{align*}
for all $n\ge 1$, where we are writing $K = K^{\qf}G(\Oc_{F,\qf})$ and $K_0(\qf) = K^{\qf}U_0(\qf)$ as in section
\ref{sec:shimura}.

Thus Corollary~\ref{cor:filtration} and the above work give the following:

\begin{theorem}\label{thm:filtration}
  There is a filtration
  \[0=V_0\subseteq V_1\subseteq V_2 \subseteq V_3 = \Mbar_{\infty,K_0(\qf)}\] of $\Rbar_\infty$-modules, with
  isomorphisms
\begin{align*}
  V_1 & \longisomto \Nbar_{\infty,K^{\qf}}, \\
  V_2/V_1 & \longisomto (\Mbar_{\infty,K})^{\oplus 2} 
            \intertext{and} 
            V_3/V_2 &\longisomto \Nbar_{\infty,K^{\qf}}.
\end{align*}
\end{theorem}

\subsection{Patching functors} \label{sec:patching-functors}

Theorem~\ref{thm:filtration} provides a link between the modules $\Mbar_{\infty,K}$ and
$\Nbar_{\infty,K^{\qf}}$. However, in order to use this to deduce properties of $\Mbar_{\infty,K}$ from those of
$\Nbar_{\infty,K^{\qf}}$ we will need additional information about the structure of $\Mbar_{\infty,K_0(\qf)}$, namely a
flatness statement for a particular submodule of $M_{\infty,K^{\qf}}(\St_\FF)\subset \Mbar_{\infty,K_0(\qf)}$.

To prove this, we will first need to introduce the notion of a \emph{patching functor},
$\sigma\mapsto M_{\infty,K}(\sigma)$. We will largely follow the presentation in \cite{MR3323575}.

We consider pairs $(S,K)$ satisfying~\ref{hyp:K}, and we take $K$ to be of the form
$K^\qf K_\qf$ for a \emph{fixed} $K^{\qf}\subseteq G(\AA^{\qf}_{F,f})$. For any $n\ge 0$ let
$K^{\qf}_n\subseteq G(\AA^{\qf}_{F,f})$ be as in section~\ref{sec:M_infty}.

We note that
\[M_{n,K}^\vee = H_1(X_{K^{\qf}_nK_{\qf}},\Oc)_{\mf, \psi^{-1}}^\vee = H^1(X_{K^{\qf}_nK_{\qf}},E/\Oc)_{\mf}[\psi]\]
for any $n\ge 0$.

Define
\[
  \Pi_{n,K^{\qf}} = \left[\dirlim_{K_{\qf}}M_{n,K^{\qf}K_{\qf}}^\vee\right]^\vee =
  \left[\dirlim_{K_{\qf}}\left[H^1(X_{K_n^{\qf}K_{\qf}},E/\Oc)_{\mf}[\psi]\right]\right]^\vee
\]
where the direct limit is taken over all compact open subgroups $K_{\qf}\subseteq G(\Oc_{F,\qf})$. Note that this
carries a continuous action of $G(\Oc_{F,\qf})\cong GL_2(\Oc_{F,\qf})$.

As the action of $\TT^S_{\mf}$ on $H^1(X_{K^{\qf}K_{\qf}},\Oc)_{\mf}[\psi]$ factors through $\TT_{K^{\qf}K_{\qf}}$, the
action of $\TT^S_{\mf}$ on $\Pi_{K^{\qf}}$ factors through
\[\TT_{K^{\qf}}= \invlim_{K_{\qf}}\TT_{K^{\qf}K_{\qf}}\]
Note that by Lemma~\ref{lem:R->T} we have natural surjections $R^\psi_{F,S}(\rhobar)\onto \TT_{K^{\qf}K_{\qf}}$ for all
$K_{\qf}$, and so we have a surjection $R^\psi_{F,S}(\rhobar)\onto \TT_{K^{\qf}_n}$.

Now following \cite{MR3323575}, let $\Cc$ be the category of finitely generated $\Oc$-modules with a continuous action
of $G(\Oc_{F,\qf})$. Let
$\psibar = (\det\rhobar|_{I_{\qf}} \bar{\epsilon})\!\circ\!\Art:\Oc_{F,\qf}^\times\to\FF^\times$ be the character
corresponding to $\det\rhobar|_{I_{\qf}}:I_{\qf}\to \FF^\times$ via local class field theory. Write
$Z = Z(G(\Oc_{F,\qf}))\cong \Oc_{F,\qf}^\times$ and let $\Cc_Z$ be the subcategory of $\Cc$ consisting of those
$\sigma\in\Cc$ possessing a central character which lifts $\psibar$ and agrees with $\psi$ on $I_{\qf}$ (in other words,
is unramified). Also let $\Cc_Z^{\fin}$ be the subcategory of finite length objects of $\Cc_Z$.

\begin{remark} In \cite{MR3323575}, the condition that the central character of $\sigma$ agrees with $\psi$ is not
  imposed; this necessitates a `twisting' argument.  We only need to patch $\sigma$ with unramified central character,
  so we avoid this technicality.
\end{remark}

Now for any $\sigma\in \Cc_Z$ and any $n\ge 0$, define
\[M_{n,K^{\qf}}(\sigma)= H^1(X_{K^{\qf}_nG(\Oc_{F,\qf})}, \Lc_{\sigma^\vee})_{\mf}[\psi]^\vee.\] For
any $\sigma$, this is a $\TT_{K^{\qf}_n}$-module, and hence an $R_n$-module. Thus we may define the
$R_n^\square$-module:
\[M_{n,K^{\qf}}^{\square}(\sigma) = M_{n,K^{\qf}}(\sigma)\otimes_{R_n}R_n^\square =
  M_{n,K^{\qf}}(\sigma)[[w_1,\ldots,w_j]].\] Now as in section~6 of \cite{MR3323575}, if $\sigma\in\Cc_Z^{\fin}$,
$\Ms_{K^{\qf}}^{\square}(\sigma)=\{M_{n,K^{\qf}}^{\square}(\sigma)\}_{n\ge 1}$ is a weak patching
$\Rs^\square$-module and thus we may define
$M_{\infty,K^{\qf}}(\sigma)=\patch(\Ms_{K^{\qf}}^{\square}(\sigma))$. We can extend this definition to all of
$\Cc_Z$ by setting
\[M_{\infty,K^{\qf}}(\sigma) = \invlim_k M_{\infty,K^{\qf}}(\sigma/\varpi^k\sigma).\] This definition
agrees with the ``patching functor'' constructed in section 6.4 of \cite{MR3323575}, up to a technicality:
the construction in \cite{MR3323575} factors out the Galois representation in the indefinite case, whereas we have not
done so. In the notation of \cite{MR3323575} the module $M_{\infty,K^{\qf}}(\sigma)$ we have constructed is
$S(\sigma)^\vee_{\mf}\cong M_\infty(\sigma)\otimes_{\TT(\sigma)_{\mf}}\rho(\sigma)_{\mf}$. However, this is simply
isomorphic to $M_\infty(\sigma)^{\oplus 2}$ as a $\TT(\sigma)_{\mf}$-module (again in the notation of \cite{MR3323575})
and so this does not present an issue.  We therefore have:

\begin{theorem}[\cite{MR3323575}]\label{thm:patching functor}
$M_{\infty,K^{\qf}}(\sigma)$ satisfies the following properties:
\begin{enumerate}
\item The functor $\sigma\mapsto M_{\infty,K^{\qf}}(\sigma)$, from $\Cc_Z$ to the category of finitely generated
  $R_\infty$-modules, is exact.
\item For any $\sigma\in\Cc_Z$, $M_{\infty,K^{\qf}}(\sigma)/\nf\cong M_{0,K^{\qf}}(\sigma)$.
\item If $\sigma\in\Cc_Z$ is a finite free $\Oc$-module, then $M_{\infty,K^{\qf}}(\sigma)$ is maximal
  Cohen--Macaulay over $R_\infty$.
\item If $\sigmabar\in\Cc_Z$ is a finite dimensional $\FF$-vector space, then $M_{\infty,K^{\qf}}(\sigmabar)$ is
  maximal Cohen--Macaulay over $\Rbar_\infty$.
\end{enumerate}
\end{theorem}

From now on assume that $\qf$ satisfies the assumptions of section~\ref{sec:ring-definitions}. That is,
$\Nm(\qf)\equiv 1\pmod{l}$, $\rhobar$ is unramified at $\qf$ and $\rhobar(\Frob_{\qf}) = \twomat{1}{1}{0}{1}$. Thus the
computations of section~\ref{sec:ring-definitions} will apply to $R_{\qf}$.
Under the map $R_{\qf}\into R_{\loc}\into R_\infty$, we may view any $R_\infty$-module as being a $R_\qf$-module.

In addition to the results listed in Theorem~\ref{thm:patching functor}, \cite{MR3323575} also describes the supports of
$M_{\infty,K^{\qf}}(\sigma)$ as $R_{\qf}$-modules, for certain $\sigma$'s corresponding to inertial types of
$F_{\qf}$. In order to avoid having to give a formal treatment of inertial types, we will simply state their results for
the specific modules $\sigma=\triv_A,\St_A$ and $\sigma^{\ps}_A$, for $A=\Oc,\FF$, defined in section~\ref{sec:types}
(noting that we have assumed that $\Oc = W(\FF)[\zeta+\zeta^{-1}]$):

\begin{proposition}[\cite{MR3323575}]\label{prop:support M_infty}
Viewing each $M_{\infty,K^{\qf}}(\sigma)$ as an $R_{\qf}$-module,
\begin{enumerate}
\item $M_{\infty,K^{\qf}}(\triv_\Oc)$ (resp. $M_{\infty,K^{\qf}}(\triv_\FF)$) is supported on $R_\qf^{\nr}$
  (resp. $\Rbar_\qf^{\nr}$),
\item $M_{\infty,K^{\qf}}(\St_\Oc)$ (resp. $M_{\infty,K^{\qf}}(\St_\FF)$) is supported on $R_\qf^{\unip}$
  (resp. $\Rbar_\qf^{\unip}$),
\item $M_{\infty,K^{\qf}}(\sigma^{\ps}_{\Oc})$ (resp. $M_{\infty,K^{\qf}}(\sigma^{\ps}_{\FF})$) is supported
  on $R_\qf^{\ps}$ (resp. $\Rbar_\qf^{\ps}$).
\end{enumerate}
\end{proposition}
\begin{proof}
  Follows from Proposition~\ref{prop:local-global} and the fact that $M_{\infty,K^{\qf}}(-)$ is a patching functor
  in the sense of \cite{MR3323575}.\qedhere
\end{proof}

We also record the support of the modules $N_{\infty,K^{\qf}}$ and $\Nbar_{\infty,K^{\qf}}$ from section~\ref{sec:M_infty} here.

\begin{proposition}\label{prop:support N_infty}
  As $R_{\qf}$-modules, $N_{\infty,K^{\qf}}$ is supported on $R^N_{\qf}$ and $\Nbar_{\infty,K^{\qf}}$ is supported on
  $\Rbar^N_{\qf}$.
\end{proposition}
\begin{proof}
  As $\Nbar_{\infty, K^{\qf}} = N_{\infty, K^{\qf}} \otimes_\Oc \FF$ and $\Rbar^N_{\qf} = R^N_\qf \otimes_\Oc \FF$, it
  suffices to prove the statement for $N_{\infty, K^{\qf}}$. 

  By the definition of $N_{\infty,K^{\qf}}$ it suffices to prove that, for any $n\ge 1$, the map
\[\gamma_n:R_{\qf}\rarrow R_n^\square\to \End_{\FF}(N_{n,K^{\qf}}^\square)\]
factors through $R_{\qf}\onto R^N_{\qf}$. We will prove this using Proposition~\ref{prop:local-global}.

Let $\TT^N_{n,K}$ be the image of $\TT^{S\cup Q_n}_\Oc$ in $\End(N_{n, K^\qf})$.  Note that the map
$R_n \rarrow \End_\Oc(N_{n, K^\qf})$ factors through $\TT^N_{n,K}$.  Define
$\TT^{N, \square}_{n,K} = \TT^N_{n,K} \otimes_{R_n} R_n^\square \cong \TT^N_{n,K}[[w_1, \ldots, w_j]]$; thus $\gamma_n$
defines a map $R^\qf \rarrow \TT^{N, \square}_{n,K}$.  Since $\TT^N_{n,K}$ is reduced and $l$-torsion free, it suffices
to show that, for every $\Oc$-algebra homomorphism
\[x : \TT^{N, \square}_{n,K} \rarrow \bar{E},\] the composition $x \circ \gamma_n$ factors through $R_{\qf}^N$.

To $x$ we have an associated homomorphism
\[\rho_x : G_{F, S\cup Q_n} \rarrow GL_2(\bar{E})\]
such that, for every $v \not \in S \cup Q_n$, $\tr(\rho_x(\Frob_v)) = x(T_v)$.  In particular, the isomorphism class of
$\rho_x$ is the Galois representation associated to $x|_{\TT_{n,K}}$.

The composition $x \circ \gamma_n$ is the homomorphism $R_{\qf} \rarrow \bar{E}$ corresponding to
$\rho_x |_{G_{F_\qf}}$.  By local-global compatibility and properties of the Jacquet--Langlands correspondence,
$\rho|_{G_{F_{\qf}}}$ is an inertially unipotent representation corresponding to a discrete series representation under
the local Langlands correspondence. It follows that $\rho|_{I_{F_{\qf}}}$ is a non-trivial unipotent representation, and
therefore that $x\circ\gamma_n:R_{\qf}\to \Ebar$ factors through $R_{\qf}\onto R_{\qf}^N$ by
Proposition~\ref{prop:local-global}. The result follows.
\end{proof}

We finish this section by relating the patching functors of this section to the patched modules $\Mbar_{\infty,K}$
considered in section~\ref{sec:M_infty}.

\begin{proposition}\label{prop:induction}
  For any compact open subgroup $K_{\qf}\subseteq G(\Oc_{F,\qf})$ we have
  \[M_{\infty,K^{\qf}}\left(\Ind_{K_{\qf}}^{G(\Oc_{F,\qf})}\triv_{\FF}\right)\cong
  \Mbar_{\infty,K^{\qf}K_{\qf}}.\] In particular, letting $K=K^{\qf}G(\Oc_{F,\qf})$ and $K_0(\qf) = K^{\qf}U_0(\qf)$,
  $\Mbar_{\infty,K}\cong M_{\infty,K^{\qf}}(\triv_{\FF})$ and
  \[\Mbar_{\infty,K_0(\qf)}\cong M_{\infty,K^{\qf}}(\triv_{\FF}\oplus \St_{\FF})\cong
    M_{\infty,K^{\qf}}(\triv_{\FF})\oplus M_{\infty,K^{\qf}}(\St_{\FF}).\]
\end{proposition}
\begin{proof}
  By the fact that $\mf$ is non-Eisenstein, we have
  \begin{align*}M_{n,K^{\qf}}\left(\Ind_{K_{\qf}}^{G(\Oc_{F,\qf})}\triv_{\FF}\right) &= \Hom_{G(\Oc_{F,\qf})}\left(
    H^1(X_{K_{\qf}G(\Oc_{F,\qf})},\FF)_\mf[\psi], \Ind_{K_{\qf}}^{G(\Oc_{F,\qf})}\triv_{\FF}\right)\\
    &=
    \Hom_{K_{\qf}}\left(
    H^1(X_{K_{\qf}G(\Oc_{F,\qf})},\FF)_\mf[\psi],
    \triv_{\FF}\right)\\
    &= \bar{M}_{n, K^{\qf}K_{\qf}}.
      \end{align*}

It follows that
$M_{n,K^{\qf}}^{\square}\left(\Ind_{K_{\qf}}^{G(\Oc_{F,\qf})}\triv_{\FF}\right) \cong
\Mbar_{n,K^{\qf}K_{\qf}}^\square$ and so
\[M_{\infty,K^{\qf}}\left(\Ind_{K_{\qf}}^{G(\Oc_{F,\qf})}\triv_{\FF}\right)=\patch\left(\Ms_{\infty,K^{\qf}}^{\square}\left(\Ind_{K_{\qf}}^{G(\Oc_{F,\qf})}\triv_{\FF}\right)
  \right)\cong \patch\left(\Msbar_{K^{\qf}K_{\qf}}\right) = \Mbar_{\infty,K^{\qf}K_{\qf}}.\] The last two statements
follow from $\Ind_{G(\Oc_{F,\qf})}^{G(\Oc_{F,\qf})}\triv_{\FF}= \triv_\FF$ and
$\Ind_{U_0(\qf)}^{G(\Oc_{F,\qf})}\triv_{\FF}= \triv_\FF\oplus\St_\FF$. The statement that
$M_{\infty,K^{\qf}}(\triv_{\FF}\oplus \St_{\FF})\cong M_{\infty,K^{\qf}}(\triv_{\FF})\oplus
M_{\infty,K^{\qf}}(\St_{\FF})$ is just a consequence of the exactness of $M_{\infty,K^{\qf}}(-)$.
\end{proof}
\begin{corollary}\label{prop:patched-filtration}
  The $R_\infty$-module 
  \[P = M_{\infty,K^{\qf}}(\triv_{\FF})\oplus M_{\infty,K^{\qf}}(\St_{\FF})\] has a filtration
  \[0=V_0\subseteq V_1\subseteq V_2\subseteq V_3=P\] with $V_1\cong V_3/V_2\cong \Nbar_{\infty,K^{\qf}}$ and
  $V_2/V_1\cong M_{\infty,K^{\qf}}(\triv_{\FF})^{\oplus 2}$.
\end{corollary}
\begin{proof}
  By Proposition~\ref{prop:induction} this is just a rephrasing of Theorem~\ref{thm:filtration}.
\end{proof}

\section{Commutative algebra lemmas}
\label{sec:commalg}

The following is a mild generalisation of the ``miracle flatness criterion'', for which see \cite{MR1011461}
Theorem~23.1 or \cite[\href{https://stacks.math.columbia.edu/tag/00R4}{Lemma 00R4}]{stacks-project}.  A
similar generalisation, in the setting of noncommutative completed group rings, also appears in \cite{1609.06965}.

\begin{lemma} \label{lem:miracle-flatness}
  Let $A \rarrow R$ be a local homomorphism of noetherian local rings, and let $M$ be a finite $R$-module.  Let $\mf$ be
  the maximal ideal of $A$.  Suppose
  that:
  \begin{enumerate}
  \item $A$ is regular;
  \item $M$ is maximal Cohen--Macaulay; and 
  \item $\dim R = \dim A + \dim R/\mf R$.
  \end{enumerate}
  Then $M$ is a flat $A$-module.
\end{lemma}

\begin{proof} The proof is essentially the same as that of \cite[\href{https://stacks.math.columbia.edu/tag/00A4}{Lemma
    00A4}]{stacks-project}.  If $M$ is zero, the result is clear; so suppose that $M$ is nonzero.  The proof is then by
  induction on $d = \dim A$.  The base case $d = 0$ is trivial, as then $A$ is a field.

  In general, suppose the lemma is true when $\dim A < d$.  Choose $x \in \mf \setminus \mf^2$.  Then $x$ is the first
  element in a regular system of parameters $(x, x_2, \ldots, x_d)$ for $A$.  The third condition implies that
  $(x, x_2, \ldots, x_d)$ extends to a system of parameters $(x, x_2, \ldots, x_d, x_{d+1}, \ldots, x_e)$ for $R$ which
  is therefore also a system of parameters for $M$ (by the hypothesis that $M$ is \emph{maximal} Cohen--Macaulay).
  Since $M$ is Cohen--Macaulay, this is a regular sequence on $M$.  In particular, $x$ is a non-zerodivisor on $M$.

  Now, $A/xA$ is regular of dimension $\dim A - 1$, $\dim(R/xR) = \dim R - 1$ (since $x$ is part of a system of
  parameters for $R$), and $M/xM$ is a maximal Cohen--Macaulay $R/xR$-module.  So, by induction, $M/xM$ is a flat
  $A/xA$-module.  Moreover, $\Tor_1^A(M, A/(x)) = 0$ as $x$ is a non-zerodivisor on $M$.  Therefore, by the local
  criterion for flatness in the form of \cite[\href{https://stacks.math.columbia.edu/tag/00ML}{Lemma
  	00ML}]{stacks-project}, $M$ is a flat $A$-module.
\end{proof}

\begin{lemma} \label{lem:steinberg-flat}
  Let $A = \FF[[X,Y]]/(X^2Y)$ and let $R$ be an $A$-algebra. Let $0 \rarrow L \rarrow M \rarrow N\rarrow 0$ be a short exact sequence of $R$-modules such that
  \begin{enumerate}
  \item $M$ is a flat $A$-module;
  \item $(X) \subset \ann_{A}(L)$;
  \item $(XY) \subset \ann_{A}(N)$.
  \end{enumerate}
  Then $N = M\otimes_{A} A/(XY)$ and so $N$ is a flat $A/(XY)$-module. Moreover we have an isomorphism $N/XN\cong L$ of $R$-modules.
\end{lemma}

\begin{proof}
  By the snake lemma, as multiplication by $X$ is zero on $L$, there is an exact sequence of $R$ modules
  \[0 \rarrow L \rarrow M[X] \rarrow N[X] \rarrow L \rarrow M/XM \rarrow N/XN \rarrow 0.\]
  But we have an exact sequence $0 \rarrow (XY) \rarrow A \rarrow (X) \rarrow 0$ (the second
  map being multiplication by $X$).  As $M$ is flat this is still exact when tensored over $A$ with $M$, and for any ideal $I$ we can
  identify $I \otimes_{A} M$ with $IM \subset M$.  Thus $M[X] = XYM$.  But as $N$ is killed by $XY$, this implies that
  the map $M[X] \rarrow
  N[X]$ is zero.  From the displayed exact sequence, we see that $L = M[X] = XYM$, and so $N = M/L = M/XYM$.  This is
  flat over $A/(XY)$.
  
  Now as $L=XYM$ and $M/XM$ is killed by $X$, the map $L\rarrow M/XM$ in the above exact sequence is zero, which implies that the map $N[X]\rarrow L$ is an isomorphism of $R$-modules.
  
  But now we have an exact sequence $0\to A/(X) \xrightarrow{\cdot Y} A/(XY)\xrightarrow{\cdot X} XA/(XY)\to 0$.
  As $N$ is flat over $A/(XY)$, the sequence of $R$-modules
  $0 \rarrow N/XN\xrightarrow{Y} N \xrightarrow{X} XN \rarrow 0$ is exact, and so we get the desired isomorphism
  $N/XN\cong N[X]\cong L$ of $R$-modules.
\end{proof}

\begin{lemma} \label{lem:switching}
  Let $B = \FF[[X,Y]]/(XY)$ and let $R$ be a complete local noetherian $B$-algebra with residue field $\FF$.
  Suppose that $L$, $M$, $N$ and $P$ are $R$-modules such that:
  \begin{enumerate}
  \item $M$ is flat over $B$;
  \item $(Y) \subset \ann_{B}(N)$ and $N$ is flat over $B/(Y)$;
  \item there is an isomorphism of $R$-modules $L \isomto M/XM$;
  \item there is an isomorphism of $R$-modules $\alpha : P \isomto L \oplus M$;
  \item there is a filtration $0 \subset P_1 \subset P_2 \subset P$ by $R$-modules and isomorphisms of $R$-modules $P_1 \isomto N$,
    $P_2/P_1 \isomto L \oplus L$, and $P/P_2 \isomto N$.
  \end{enumerate}
  Then there is a short exact sequence of $R$-modules
  \[0 \rarrow N \rarrow M/Y \rarrow N \rarrow 0.\]
\end{lemma}

\begin{proof}
  Since $L$ is flat over $B/X$ by points (1) and (3), it has no $Y$-torsion, and so $\alpha$ induces an isomorphism
  $P[Y] \isomto M[Y]$.  From the short exact sequence
  \[0 \rarrow P_1 \cong N \rarrow P_2 \rarrow L\oplus L \rarrow 0\] of point (5), we have $P_2[Y] = P_1[Y] \cong N$.

  From the other short exact sequence
  \[0 \rarrow P_2 \rarrow P \rarrow N \rarrow 0 \] of point (5), we get an exact sequence
  \[0 \rarrow P_2[Y] \cong N \rarrow P[Y] \cong M[Y] \rarrow N[Y] = N\] By the flatness of $M$, we can identify $M[Y]$
  with $X\cdot M$, and so the image of $M[Y]$ in $N$ is $XN$.  Since $N$ is flat over $B/(Y)$, $N \cong XN$.  Thus we
  have a short exact sequence
  \[0 \rarrow P_2[Y] \cong N \rarrow M[Y] \cong XM \rarrow X\cdot N\cong N \rarrow 0.\]

  Finally, since $M$ is flat over $B$ there is an isomorphism $M/YM \isomto XM$.  We get the desired short exact
  sequence:
  \[0 \rarrow N \rarrow XM \cong M/YM \rarrow N \rarrow 0.\qedhere\]
\end{proof}

\section{Ihara's lemma}\label{sec:ihara} Let $D$ be a quaternion division algebra over $F$ ramified at exactly one infinite place, so that we are in the indefinite case
of section~\ref{sec:shimura}.  Suppose that $\pf$ is a finite place of $F$ at which $D$ is unramified.

\subsection{Statements}\label{sec:statements} Let $K\subseteq G(\AA_{F,f})$ be unramified at $\pf$ and sufficiently
small, and let $S$ be any finite set of finite places of $F$ containing $\Sigma(K)\cup \Sigma_l\cup \{\pf\} \cup
\Sigma_\infty$.  
There are two natural degeneracy maps $\pi_1,\pi_2:X_{K_0(\pf)}\to X_K$, defined in section~\ref{sec:degeneracy}.

\begin{conjecture}\label{Ihara arbitrary weight}
  Suppose that $\Lambda$ is the local system on $X_K$ attached to a finite-dimensional continuous
  $\FF_l$-representation of $K^{\pf}$. Then for any non-Eisenstein maximal ideal $\mf$ of $\TT^S_{\ZZ_l}$ the map
\[\pi_1^*\oplus\pi_2^*:H^1_{\et}(X_K,\Lambda)_\mf\oplus H^1_{\et}(X_K,\Lambda)_\mf \to H^1_{\et}(X_{K_0(\pf)},\Lambda)_\mf\]
is injective.
\end{conjecture}

For $\Lambda$ the constant sheaf $\FF_l$, this becomes:

\begin{conjecture}\label{Ihara weight 2}
For any non-Eisenstein maximal ideal $\mf$ of $\TT^S_{\ZZ_l}$, the map
\[\pi_1^*\oplus\pi_2^*:H^1_{\et}(X_K,\FF_l)_\mf\oplus H^1_{\et}(X_K,\FF_l)_\mf \to H^1_{\et}(X_{K_0(\pf)},\FF_l)_\mf\]
is injective.
\end{conjecture}

We also have an equivalent dualized version:

\begin{conjecture}\label{Ihara surjective}
  For any non-Eisenstein maximal ideal $\mf$ of $\TT^S_{\ZZ_l}$, the map
\[(\pi_{1,*},\pi_{2,*}):H_1(X_{K_0(\pf)},\FF_l)_\mf\to H_1(X_K,\FF_l)_\mf\oplus H_1(X_K,\FF_l)_\mf\]
is surjective.
\end{conjecture}

\begin{lemma}
  Conjecture~\ref{Ihara weight 2} (or, equivalently, Conjecture~\ref{Ihara surjective}) for all $K$ implies
  Conjecture~\ref{Ihara arbitrary weight} for all $K$.
\end{lemma}
\begin{proof}
  Suppose that Conjecture~\ref{Ihara weight 2} holds for all $K$.  Suppose that $\Lambda$ and $\mf$ are as in the
  statement of Conjecture~\ref{Ihara arbitrary weight}, and that $\Lambda$ is associated to a representation $V$ of
  $K^{\pf}$.  Let $H^\pf \subset K^\pf$ be an open subgroup that acts trivially on $V$, and $H = H^\pf K_\pf$.  Let
  $f : X_H \rarrow X_K$ be the projection.  The Hochschild--Serre spectral sequence provides a (Hecke-equivariant) exact
  sequence
  \[ 0 \rarrow H^1(K/H, H^0(X_H, f^*\Lambda)) \rarrow H^1(X_K, \Lambda) \rarrow H^0(K/H, H^1(X_H, f^*\Lambda)).\]
  After localizing at $\mf$, the first term vanishes by Lemma~\ref{prop:02eisenstein}.  Noting that $f^*\Lambda$ is
  constant, we get an inclusion
  \[H^1(X_K, \Lambda)_\mf \into H^1(X_H, \FF_l^{\dim V})_\mf\] that commutes with the maps $\pi^*$.  Since
  Conjecture~\ref{Ihara weight 2} holds for the subgroup $H$ by assumption, we deduce Conjecture~\ref{Ihara arbitrary
    weight} for the subgroup $K$.
\end{proof}
Our main result is the following:
\begin{theorem}\label{main thm}
  Conjectures~\ref{Ihara arbitrary weight}, \ref{Ihara weight 2} and~\ref{Ihara surjective} are true for any
  non-Eisenstein maximal ideal $\mf$ of $\TT^S_{\ZZ_l}$ satisfying the conditions:
\begin{enumerate}
\item $l|\#\rhobar_{\mf}(G_F)$. That is, $\mf$ is not exceptional.
\item If $l = 5$ and the image of the projective representation
  $\proj \rhobar_{\mf}:G_F\to GL_2(\FFbar_5)\onto PGL_2(\FFbar_5)$ is isomorphic to $PGL_2(\FF_5)$, then
  $\ker\proj \rhobar_{\mf}\not\subseteq G_{F(\zeta_5)}$. (This condition is automatically satisfied whenever
  $\sqrt{5}\not\in F$.)
\end{enumerate}
\end{theorem}

\begin{remark}
  Condition (1) implies the Taylor--Wiles condition that $\rhobar_{\mf}|_{G_{F(\zeta_l)}}$ is absolutely
  irreducible. Condition (2) is simply the other Taylor--Wiles condition (see
  \cite[3.2.3]{Kisin2009-ModuliFFGSandModularity}).

  The reason for including the stronger assumption that $\mf$ is not exceptional, instead of just the usual
  Taylor--Wiles conditions, is that this assumption will be necessary for picking the auxiliary prime $\qf$. See
  Lemma~\ref{lem:auxiliary} below.
\end{remark}

\begin{remark}
  We have assumed that $K$ is sufficiently small, for convenience.  This assumption could be removed by the standard
  device of introducing auxiliary level structure at a place $\qf_0$ at which there are no congruences, as in
  \cite{Manning} section~4.2 or \cite{MR3323575} section~6.2.
\end{remark}

\subsection{Definite quaternion algebras}
\label{sec:definite}
Let $\bar{D}$ be a totally definite quaternion algebra over $F$, unramified at $\pf$.  Let $\bar{G}$ be the associated
algebraic group. If $H \subset \bar{G}(\AA_{F,f})$ is a compact open subgroup unramified at $\pf$ then we have
degeneracy maps $\pi_1,\pi_2:Y_{H_0(\pf)}\to Y_{H}$.  Let $S$ be a finite set of places of $F$ containing
$\Sigma_l \cup \Sigma_\infty \cup \{\pf\}$ and all places at which $H$ or $\bar{D}$ ramify. The following version of
Ihara's Lemma is known:

\begin{theorem}\label{Ihara definite}
  If $H\subseteq \bar{G}(\AA_{F,f})$ is unramified at $\pf$, then for any non-Eisenstein maximal ideal $\mf$ of $\TT_{\ZZ_l}^S$, the map
\[\pi^* = \pi_{1}^*+\pi_{2}^*: H^0(Y_{H},\FF_l)_\mf\oplus H^0(Y_{H},\FF_l)_\mf\rarrow H^0(Y_{H_0(\pf)},\FF_l)_\mf\]
is injective.
\end{theorem}
\begin{proof} Versions of this have been proved by Ribet (over $\QQ$, \cite{MR1047143} Theorem~3.15) and Taylor (over
  $F$, \cite{MR1016264} Lemma~4).  There it is proved that with $\ZZ_l$ coefficients, without localizing at $\mf$,
  $\pi^*$ has saturated image, from which the theorem may be easily deduced --- but the method for doing this actually directly
  gives the result in the form we need.  For $\QQ$ this is carried out in \cite{MR1262939} Lemma~2 and the general case
  is no harder.  We include the proof for completeness.

  Suppose that $(f,g)$ is in the kernel of $\pi^*$.  Regard $f$ and $g$ as $H$-invariant functions on
  $\bar{G}(F) \backslash \bar{G}(\AA_{F,f})$.  Then $f(x) = -g(x\omega)$ for all $x$ in this quotient, where
  $\omega = \twomat{\varpi_\pf}{0}{0}{1}$ (making use of the isomorphism $\bar{G}(F_\pf) \cong GL_2(F_\pf)$).  Then $f$
  is invariant under $H$ and $\omega^{-1}H\omega$.  These subgroups generate a subgroup containing $H^\pf SL_2(F_\pf)$,
  under which $f$ is invariant.  Let $\bar{G}'$ be the subgroup of $\bar{G}$ of elements with reduced norm 1.  Then by
  the strong approximation theorem in $\bar{G}'$, the function $f$ factors through the reduced norm map:
  \[ \nu :\bar{G}(F) \backslash \bar{G}(\AA_{F,f})/\bar{G}'(\AA_{F,f})H \rarrow F^\times \backslash \AA_{F,f}^\times /\nu(H).\]
  But the functions factoring through this map form a module over $\TT_H^S$ that is supported on Eisenstein maximal
  ideals (the argument is similar to that of Proposition~\ref{prop:02eisenstein}).  The theorem follows.
\end{proof}

\subsection{The auxiliary prime}\label{sec:auxiliary}
Recall our assumption that $l|\#\rhobar_\mf(G_F)$. After conjugating $\rhobar_\mf$ if necessary, we may thus assume that $\rhobar_\mf(G_F)$ contains the matrix
$\twomat{1}{1}{0}{1}$. We now get the following:

\begin{lemma}\label{lem:auxiliary}
There are infinitely many primes $\qf$ for which:
\begin{enumerate}
	\item $\qf\not\in \Delta\cup\Sigma(K)\cup \Sigma_l\cup \{\pf\}$
	\item $\rhobar_\mf$ is unramified at $\qf$
	\item $\Nm(\qf)\equiv 1\pmod{l}$
	\item $\rhobar_\mf(\Frob_\qf) = \twomat{1}{1}{0}{1}$
\end{enumerate}
\end{lemma}
\begin{proof}
All but finitely many primes satisfy (1) and (2), so it suffices to find infinitely many primes satisfying (3) and (4).

Pick a number field $L/F$ for which $F(\zeta_l)\subseteq L$ and $\rhobar_\mf:G_F\to GL_2(\bar{\FF}_l)$ factors
through $\Gal(L/F)$. Let $\bar{\epsilon}:\Gal(L/F)\onto\Gal(F(\zeta_\ell)/F)\into (\ZZ/l\ZZ)^\times$ be the
cyclotomic character. By the Chebotarev density theorem, it suffices to find some $\sigma\in \Gal(L/F)$ for which
$\rhobar_\mf(\sigma) = \twomat{1}{1}{0}{1}$ and $\bar{\epsilon}(\sigma) = 1 \in (\ZZ/l\ZZ)^\times$.

Now by our assumption on the image of $\rhobar_\mf$, there is some $\sigma_0\in \Gal(L/F)$ for which
$\rhobar_\mf(\sigma_0) = \twomat{1}{1}{0}{1}$. Let $\sigma = \sigma_0^{1-l}\in \Gal(L/F)$. Then we indeed have
\[\rhobar_\mf(\sigma) = \rhobar_\mf(\sigma_0)^{1-l} = \twomat{1}{1}{0}{1}^{1-l} = \twomat{1}{1}{0}{1}\]
and
\[\bar{\epsilon}(\sigma) = \bar{\epsilon}(\sigma_0)^{l-1} = 1\in (\ZZ/l\ZZ)^\times.\qedhere\]
\end{proof}

For the rest of the proof we fix such a prime $\qf$.  Note that it satisfies the requirements of
sections~\ref{sec:integral} and~\ref{sec:ring-definitions}.  We let $\bar{D}$ be a definite quaternion algebra ramified
at $\Delta \cup \{\qf, \tau\}$.

\subsection{The proof}\label{sec:proof}

Choose $\FF$ large enough that $\rhobar_\mf$ is defined over $\FF$, and let $(E, \Oc, \FF)$ be the coefficient system
satisfying Hypothesis~\ref{hyp:coeffs}.  Let $\psi : G_F \rarrow \Oc^\times$ be a finite order character lifting
$\det(\rhobar_\mf) \bar{\epsilon}$, and also write $\psi$ for the character $\psi \circ \Art_F$ of
$\AA_{F,f}^\times/F^\times$.  We make sure that $F^\times(K \cap Z(\AA_{F,f})) \subset \ker(\psi)$ and that the prime
$\qf$ is chosen so that $\psi$ is unramified at $\qf$.

Let $S$ be as in section~\ref{sec:statements}. Enlarging $S$ if necessary (which is allowed, by Lemma~\ref{lem:R->T}),
we assume that $\qf \in S$.  We write $\Sigma$ for the set of finite places in $S$. The results of
section~\ref{sec:shimura} imply that there is a filtration of $H^1(X_{K_0(\qf)}, \FF)_\mf[\psi]$ (by $\TT^S$-submodules)
whose graded pieces are
\[H^0(Y_{K^\qf}, \FF)_\mf[\psi], H^1(X_K, \FF)_\mf[\psi]^{\oplus 2}, H^0(Y_{K^\qf}, \FF)_\mf[\psi].\]

In section~\ref{sec:patching} we explain how these cohomology groups and this filtration (more precisely, their duals) may be `patched' using the
Taylor--Wiles method.  For each place $v \in \Sigma$ let $R_v$ be
\begin{itemize}
\item if $v \nmid l$, the universal fixed determinant framed deformation ring
  $R^{\square, \psi}_{\rhobar_\mf|_{G_{F_v}}, \Oc}$ of $\rhobar_\mf|_{G_{F_v}}$;
\item if $v \mid l$, the potentially semistable (over a fixed extension depending only on $K \cap G(F_v)$, and of parallel
  Hodge--Tate weights $\{0,1\}$) deformation ring $R^{\square, \psi, K\cap G(F_v)\text{-st}}_{\rhobar_\mf|_{G_{F_v}},\Oc}$
  defined in section~\ref{sec:local-global}.
\end{itemize}

For some integers $g,d \geq 0$ (determined in section~\ref{sec:patching}, with $d=r+j$ in the notation of that section) we let
\[R_\infty = \left(\widehat{\otimes}_{v \in \Sigma} R_v\right)[[X_1, \ldots, X_g]]\] and
\[S_\infty = \Oc[[Y_1, \ldots, Y_d]],\] 
and recall that $d$ and $g$ were chosen so that $R_\infty$ and $S_\infty$ have the same dimension.

Then in section~\ref{sec:patching-functors} we constructed an injective homomorphism $S_\infty \rarrow R_\infty$,
maximal Cohen--Macaulay $R_\infty$-modules $M_{\infty,K}$ and $N_{\infty, K^\qf}$, and an exact functor
$M_{\infty, K^\qf}$ from the category of finitely-generated $\Oc$-modules with a continuous action of
$GL_2(\Oc_{F,\qf})$ (satisfying a condition on the central character) to the category of finitely-generated
$R_\infty$-modules.  Moreover, $M_{\infty, K^\qf}$ has the property that if $\sigma$ is a finite free $\Oc$-module
(resp. a finite dimensional $\FF$-vector space) then $M_{\infty, K^\qf}(\sigma)$ is maximal Cohen--Macaulay over
$R_\infty$ (resp. $\bar{R}_\infty = R_\infty\otimes_{\Oc}\FF$). These are equipped with isomorphisms
\[M_{\infty, K} \otimes_{S_{\infty}} \FF \cong H_1(X_K, \FF)_{\mf, \psi^{-1}}\]
and
\[N_{\infty, K^{\qf}}\otimes_{S_{\infty}} \FF \cong H_0(Y_{K^\qf}, \FF)_{\mf, \psi^{-1}}.\]

In Table~\ref{tab:patched-supports}, for various patched modules, we write down a corresponding quotient $R^{?}_{\qf}$
of $R_\qf$ on which they are supported.
\begin{table}[hbt] 
 \caption{Supports of patched modules.\label{tab:patched-supports}}
\begin{tabular}{ll}
  patched module $M_\infty$ & quotient $R^{?}_{\qf}$ \\
  \hline \\[-1.0em]
  $M_{\infty,K^\qf}(\triv_{\Oc})$ & $R_\qf^{\nr}$\\
  $M_{\infty,K^\qf}(\St_\Oc)$ & $R_\qf^{\unip}$\\
  $N_{\infty,K^\qf}$ & $R_\qf^{N}$\\
  $M_{\infty,K^\qf}(\sigma^{\ps}_{\Oc})$ & $R_\qf^{\ps}$\\
\end{tabular}
\end{table}
Here $?$ is an element of
$\{\nr, N, \unip, \ps\}$, and we write
\[R^?_\qf = R^{?,\psi}_{\rhobar|_{G_{F_\qf}}, \Oc}\] 
and
\[\bar{R}^?_\qf = R^?_\qf\otimes_{\Oc}\FF,\]
as shorthand for the rings defined in section~\ref{sec:ring-definitions}.  The claims of
Table~\ref{tab:patched-supports} follow from the properties of the Jacquet--Langlands correspondence and local-global
compatibility, and are the content of Propositions~\ref{prop:support M_infty} and~\ref{prop:support N_infty}.  Furthermore, for
$? \in \{\nr, N, \unip, \ps\}$ we define the quotient
\[R^?_\infty = R_\qf^? \widehat{\otimes}\left(\widehat{\otimes}_{v \in \Sigma\setminus\{\qf\}} R_v\right)[[X_1, \ldots, X_g]]\]
of $R_\infty$.

The filtration provided by Theorem~\ref{thm:vanishing} may be patched as in section~\ref{sec:patching}.  Thus (see
Corollary~\ref{prop:patched-filtration}) there is a filtration of
\[P = \Mbar_{\infty, K_0(\qf)} = M_{\infty,K^{\qf}}(\triv_\FF) \oplus M_{\infty,K^{\qf}}(\St_\FF)\]
by $R_\infty$-modules
\begin{equation}0 \subset P_1 \subset P_2 \subset P \tag{$\star$}\label{eq:filtration}\end{equation} together with isomorphisms
\begin{align*}\Nbar_{\infty,K^{\qf}} &\isomto P_1,\\
 \Nbar_{\infty,K^{\qf}} &\isomto P/P_2, \\
\intertext{and} M_{\infty, K^{\qf}}(\triv_\FF)^{\oplus 2} &\isomto P_2/P_1.\end{align*}

To go further, we need the structure of the local deformation rings at $\qf$.  The deformation rings
$R^{\nr}_{\qf}$, $R^N_\qf$ and $R^{\ps}_\qf$ are regular by Propositions~\ref{prop:ps-ring} and~\ref{prop:unip-ring}.
Therefore, by Lemma~\ref{lem:miracle-flatness}, we have:
\begin{proposition} \label{prop:flat-modules}
  \begin{enumerate}
  \item $M_{\infty, K^{\qf}}(\triv_\Oc)$ is flat over $R_\qf^{\nr}$. \label{a}
  \item $N_{\infty, K^{\qf}}$ is flat over $R^N_\qf$. \label{b}
  \item $M_{\infty, K^{\qf}}(\sigma^{\ps}_\Oc)$ is flat over $R^{\ps}_\qf$. \label{c}
  \end{enumerate}
\end{proposition}
By Proposition~\ref{prop:ring-comparison}, there are isomorphisms
\[\Rbar^{\unip}_\qf \isomto \FF[[X,Y,P,Q,R]]/(XY)\]
and
\[\Rbar^{\ps}_\qf \isomto \FF[[X, Y, P,Q,R]]/(X^2Y)\]
compatible with the natural surjection $\Rbar^{\ps}_\qf  \onto \Rbar^{\unip}_\qf$ and so that
\begin{align*}\Rbar_\qf^{\nr} &= \Rbar_\qf^{\unip}/(X)\\ \intertext{and} \Rbar_\qf^{N} &= \Rbar_\qf^{\unip}/(Y).\end{align*}

By section~\ref{sec:types}, equation~\eqref{ps-ses}, we have an exact sequence
\[0 \rarrow M_{\infty, K^{\qf}}(\triv_\FF) \rarrow M_{\infty, K^{\qf}}(\sigma^{\ps}_{\FF}) \rarrow M_{\infty, K^{\qf}}(\St_\FF) \rarrow 0.\]
\begin{proposition} \label{prop:steinberg-flat} The module $M_{\infty,K}(\St_\FF)$ is flat over $R_\qf^{\unip}$ and there
  is an isomorphism
  \[M_{\infty, K^{\qf}}(\St_\FF) \otimes_{\Rbar_\qf^{\unip}} \Rbar_\qf^{\nr} \isomto M_{\infty, K^{\qf}}(\triv_\FF)\cong \Mbar_{\infty,K^{\qf}}.\]
\end{proposition}
\begin{proof}
  By Proposition~\ref{prop:flat-modules} and the above exact sequence, the hypotheses of Lemma~\ref{lem:steinberg-flat}
  apply with $R = R^{\unip}_\infty \otimes_\Oc \FF$ (made into an $\FF[[X,Y]]/(X^2Y)$-algebra in the evident way),
  $L = M_{\infty, K^{\qf}}(\triv_\FF)$, $M = M_{\infty, K^{\qf}}(\sigma^{\ps}_\FF)$, and
  $N = M_{\infty, K^{\qf}}(\St_\FF)$.  The proposition follows.
\end{proof}

Now we know that $M_{\infty, K^{\qf}}(\St_\FF)$ is flat, the filtration~(\ref{eq:filtration}) can be used to ``transfer
information'' between $N_\infty$ and $M_\infty$.  More precisely, we have:

\begin{proposition} \label{prop:switching}
  There is a short exact sequence of $R_\infty$-modules
  \[0 \rarrow \Nbar_{\infty,K^{\qf}} \rarrow M_{\infty, K^{\qf}}(\St_\FF) \otimes_{\Rbar_\qf^{\unip}} \Rbar_\qf^{N} \rarrow \Nbar_{\infty,K^{\qf}} \rarrow 0.\]
\end{proposition}
\begin{proof}
  By Proposition~\ref{prop:steinberg-flat} and the filtration~(\ref{eq:filtration}), the hypotheses of
  Lemma~\ref{lem:switching} apply with $R= R^{\ps}_\infty \otimes_{\Oc} \FF$ (made into an $\FF[[X,Y]]/(XY)$-algebra in
  the evident way), $L = M_{\infty,K^{\qf}}(\triv_\FF)$, $M = M_{\infty,K^{\qf}}(\St_{\FF})$,
  $P = \Mbar_{\infty, K_0(\qf)}$, $N = \Nbar_{\infty, K^{\qf}}$, and $P_1$ and $P_2$ given by~$\eqref{eq:filtration}$.  The proposition follows.
\end{proof}

\begin{proof}[Proof of Theorem~\ref{main thm}:] Now we are ready to prove our main result.  We may carry out the
  constructions and arguments above equally well with $K^\qf$ replaced by $K_0(\pf)^\qf$ in a way compatible with the
  degeneracy maps $\pi_*$.  We therefore obtain a commuting diagram
  \begin{equation*}
    \begin{CD}
      0 @>>> \Nbar_{\infty, K_0(\pf)^{\qf}} @>>> M_{\infty, K_0(\pf)^{\qf}}(\St_\FF)/(Y)@>>> \Nbar_{\infty, K_0(\pf)^{\qf}}@>>> 0\\
      @. @V{\pi_*}VV @V{\pi_*}VV  @V{\pi_*}VV @. \\
      0 @>>> \left(\Nbar_{\infty, K^{\qf}}\right)^{\oplus 2} @>>> M_{\infty, K^{\qf}}(\St_\FF)^{\oplus 2}/(Y)@>>> \left(\Nbar_{\infty, K^{\qf}}\right)^{\oplus 2}@>>> 0.\\
    \end{CD}
  \end{equation*}
  By Theorem~\ref{Ihara definite} the outer maps are surjective after applying $\otimes_{S_\infty} \FF$, and so by
  Nakayama's Lemma they are surjective.  It follows that the middle map is surjective, and by Nakayama's Lemma again
  that the map
  \[\pi_* : M_{\infty, K_0(\pf)^\qf}(\St_\FF) \rarrow M_{\infty, K^\qf}(\St_\FF)^{\oplus 2}\]
  is surjective. Tensoring with $\Rbar^{\nr}_{\qf}$ and applying Proposition~\ref{prop:steinberg-flat} this gives that
  $\Mbar_{\infty,K_0(\pf)}\rarrow \left(\Mbar_{\infty,K}\right)^{\oplus 2}$ is surjective. Applying
  $\otimes_{S_\infty} \FF$, we see that
  \[\pi_* : H_1(X_{K_0(\pf)}, \FF)_{\mf, \psi^{-1}} \rarrow H_1(X_K, \FF)_{\mf, \psi^{-1}}^{\oplus 2}\]
  is surjective.  By Nakayama's Lemma, we obtain that 
  \[\pi_* : H_1(X_{K_0(\pf)}, \FF)_{\mf} \rarrow H_1(X_K, \FF)_{\mf}^{\oplus 2}\]
  is surjective.  This proves Conjecture~\ref{Ihara surjective} and hence Theorem~\ref{main thm} for this
  $\mf$.
  \end{proof}

\bibliography{refs}{}
\bibliographystyle{amsalpha}

\end{document}